\documentclass{amsart}
\usepackage{amssymb}
\usepackage{verbatim}
\usepackage{array}
\usepackage{latexsym}
\usepackage{enumerate}
\usepackage{amsmath}
\usepackage{amsfonts}
\usepackage{amsthm}
\usepackage{color}
\usepackage[english]{babel}
\usepackage{comment}

\newtheorem{theorem}{Theorem}[section]
\newtheorem{proposition}[theorem]{Proposition}
\newtheorem{lemma}[theorem]{Lemma}
\newtheorem{corollary}[theorem]{Corollary}

\newtheorem{result}[theorem]{Result}

\newtheorem{rem}[theorem]{Remark}

\def\cC{\mathcal C}
\def\cD{\mathcal D}

\def\cF{\mathcal F}

\def\cH{\mathcal H}

\def\cL{\mathcal L}

\def\cQ{\mathcal Q}

\def\cX{\mathcal X}

\def\Aut{\mbox{\rm Aut}}

\def\PG{{\rm{PG}}}

\def\deg{\mbox{\rm deg}}
\def\det{\mbox{\rm det}}

\def\Aut{\mbox{\rm Aut}}

\def\Div{\mbox{\rm div}}

{\def\Sym{\mbox{\rm Sym}}
\def\dim{\mbox{\rm dim}}

\def\gg{\mathfrak{g}}

\newcommand{\PGL}{\mbox{\rm PGL}}

\newcommand{\aut}{\mbox{\rm Aut}}

\newcommand{\g}{\gamma}

\newcommand{\gz}{\zeta}

\newcommand{\ha}{{\textstyle\frac{1}{2}}}

\title{A generalization of Bring's curve in any characteristic}
\date{}
\begin{small}

\author[G. Korchm\'aros]{G\'abor Korchm\'aros}
\author[S. Lia]{Stefano Lia}
\author[M. Timpanella]{Marco Timpanella}
\address{G\'abor Korchm\'aros, \textnormal{Dipartimento di
Matematica, Informatica ed Economia Universit\`a degli Studi  della Basilicata.}}
\email{gabor.korchmaros@unibas.it}
\address{Stefano Lia,
\textnormal{School of Mathematics and Statistics, University College Dublin, Belfield, Ireland.}}
\email{stefano.lia@ucd.ie}
\address{Marco Timpanella,
\textnormal{Dipartimento di Matematica e Informatica, Università di Perugia.}}
\email{marco.timpanella@unipg.it}
\end{small}
\begin{document}


\begin{abstract}  A natural generalization of Bring's curve valid over any field $\mathbb{K}$ of zero characteristic or positive characteristic $p\ge 7$, is the algebraic variety $V$ of $\textrm{PG}(m-1,\mathbb{K})$, 
$m\ge 5$,  which is the complete intersection of the projective algebraic hypersurfaces of homogeneous equations $x_1^k+\cdots +x_m^{k}=0$ with $1\leq k\leq m-2$. In positive characteristic, we also assume $m\le p-1$.
 Up to a change of coordinates in $\textrm{PG}(m-1,\mathbb{K})$, we show that $V$ is a projective, absolutely irreducible, non-singular curve of $\textrm{PG}(m-2,\mathbb{K})$.
We show that, if the automorphism group $\aut(V)$ is tame (in particular in characteristic zero), then $\aut(V)$ is linear and isomorphic to $\textrm{Sym}_m$. Remarkably, in positive characteristic, $\aut(V)$ may be larger; in particular $\aut(V)$ happens to  contain non-tame projectivities and this occurs precisely when $m=p-1$. 
In positive characteristic, we further investigate the two extremal cases,  $m=5$ and $m=p-1$, providing information on the number of rational points. In particular, for $m=5$ we show that there exist infinitely many primes $p$ such that $V$ is an $\mathbb{F}_{p^2}$-maximal curve of genus $4$.
We also highlight connections with previous work by R\'edei on the famous Minkowski conjecture proven by Haj\'os (1941), as well as with more recent results by Rodríguez Villegas, Voloch, and Zagier (2001) on plane curves attaining the St\"ohr-Voloch bound, and the regular sequence problem for systems of diagonal equations introduced by Conca, Krattenthaler, and Watanabe (2009).
\end{abstract}
\maketitle
\vspace{0.5cm}\noindent {\em Keywords}:
algebraic curves, function fields, positive characteristic, automorphism groups.
\vspace{0.2cm}\noindent

\vspace{0.5cm}\noindent {\em Subject classifications}:
\vspace{0.2cm}\noindent  14H37, 14H05.

\section{Introduction} Bring's curve is well known from classical geometry as being the curve with the largest automorphism group among all genus $4$ complex curves; see \cite{BN}. Its canonical representation in the complex $4$-space is the complete intersection of three algebraic hypersurfaces of equation $x_1^k+\cdots +x_5^k=0$ with $1\leq k\leq 3$. 
A natural generalization of Bring's curve valid over any field $\mathbb{K}$ of zero characteristic, or positive characteristic $p\ge 7$, is the algebraic variety $V$ of $\textrm{PG}(m-1,\mathbb{K})$, $m\ge 5$, which is the complete intersection of the projective algebraic hypersurfaces of homogeneous equations
\begin{equation}
\label{sy}\left\{
\begin{array}{llll}
X_1 + X_2 + \ldots + X_m=0;\\
X_1^2  +  X_2^2 + \ldots + X_m^2=0;\\
\cdots\cdots\\
\cdots\cdots\\
X_1^{m-2}  +  X_2^{m-2}  +  \ldots+X_m^{m-2}=0.
\end{array}
\right.
\end{equation}

From now on we assume $m\le p-1$ when $\mathbb{K}$ has characteristic $p$. The first equation in (\ref{sy}) implies that $V$ is contained in the hyperplane of homogeneous equation $X_1 + X_2 + \ldots + X_m=0$.
Up to a change of coordinates in $\textrm{PG}(m-1,\mathbb{K})$, we show that $V$ is a projective, absolutely irreducible, non-singular curve of degree $(m-2)!$ embedded in $\textrm{PG}(m-2,\mathbb{K})$; see Theorem \ref{the2502}.

The symmetric group ${\rm{Sym}}_m$ has a natural action on the coordinates $(X_1:\ldots: X_m)$ of $\textrm{PG}(m-1,\mathbb{K})$. Therefore, the automorphism group $\aut(V)$ of $V$ has a subgroup $G$ isomorphic to ${\rm{Sym}}_m$. We prove that $G=\aut(V)$ whenever $\aut(V)$ is tame, and in particular in zero characteristic; see Theorem \ref{the171021}. In Section \ref{nta} we show however that if $p=m-1$ then $\aut(V)$ is non-tame, i.e. $V$ has an automorphism of order $p$ fixing a point of $V$. In particular, the linear subgroup $L\aut(V)$ of $\aut(V)$ is isomorphic to $S_p$ and hence $\aut(V)$ properly contains $G$; see Theorem \ref{th26062024A} whose proof uses the deep Zalesski$\breve{i}$-Sere$\hat{z}$kind classification theorem from Group theory. 
In Section \ref{secqc} we investigate the quotient curves of $V$ with respect to subgroups of $G$. The most interesting cases for the choice of a  subgroup are the stabilizers of one or more coordinates. Since two such subgroups are conjugate in $G$ if they fix the same number of coordinates, it is enough to consider the stabilizer $G_{d+1,\ldots,m}$ of the coordinates $X_{d+1},\ldots,X_{m}$ in $\aut(V)$ with $1\le d \le m-1$. A careful analysis of the actions of these stabilizers on the points of $V$, carried out in Section \ref{secqc}, together with the Hurwitz genus formula, allows to compute the genus of the quotient curve $V_d=V/G_{d+1,\ldots,m}$; see Proposition \ref{pro11giugno}. The quotient curve $V_d$ turns out to be rational for $d=m-2$, otherwise its genus $\mathfrak{g}$ is quite large; see (\ref{eq11giugno}).

The projection of $V$ from the $m-d-1$ dimensional projective subspace $\Sigma$ of equations $X_1=0,\ldots,X_d=0$ is also a useful tool for the study of $V$ by virtue of the fact that $\Sigma\cap V=\emptyset$.
Let $\bar{V}_d$ be the curve obtained by projecting $V$ from the vertex $\Sigma$ to a projective subspace $\Sigma'$ of dimension $d-1$ disjoint from $\Sigma$. Comparison of $\bar{V}_d$ with the quotient curve $V_d$ shows that they are actually isomorphic; see Proposition \ref{pro5mar}. Therefore, the function field extension $\mathbb{K}(V):\mathbb{K}(\bar{V}_d)$ is Galois, that is, the vertex $\Sigma$ can be viewed as a Galois subspace (also called higher dimensional Galois-point) for $V$. This shows that there are many higher dimensional Galois-subspaces for $V$, which is a somewhat rare phenomena. For results and open problems on Galois-points; see the recent papers \cite{AR,Fu,FH,H,KLT}. In Section \ref{sq} the case $d=m-3$ is investigated more closely. An explicit equation of $V_{m-3}$ is given in Theorem \ref{the24jun}, which shows that if $\mathbb{K}$ is either an algebraic closure of the rational field $\mathbb{Q}$, or of the prime field $\mathbb{F}_p$, then $V_{m-3}$ coincides with the plane curve of degree $m-2$ investigated by Rodr\'iguez Villegas, Voloch and  Zagier \cite{voloch}, who pointed out that this curve has many points; in particular, for $m=p-1$, it is a non-singular plane curve attaining the St\"ohr-Voloch bound.

From an algebraic number theory point of view, (\ref{sy}) is a particular system of diagonal equations. Proposition \ref{prop10.03} shows that every solution of (\ref{sy}) also satisfies further diagonal equations.
On the other hand, Lemmas \ref{lem21jun21} and \ref{lem21jun21bis} imply that $\{1,2,\ldots,m\}$ provides a regular sequence of symmetric polynomials, i.e. the system of diagonal equations of $x_1^k+\cdots +x_m^k=0$ with $k$ ranging over $\{1,2,\ldots, m\}$ has only the trivial solution $(0,0,\ldots,0)$.
This gives a different proof for Proposition 2.9 of the paper of Conca, Krattenthaler and Watanabe \cite{CKW} where the authors relied on the interpretation of the $q$-analogue of the binomial coefficient as a Hilbert function. In the same paper \cite{CKW}, the general notion of a regular sequence associated to  $\mathcal{A}\subset \mathbb{N}^*$ was introduced as any system of diagonal equations $x_1^k+\cdots +x_m^{k}=0$
with $k\in \mathcal{A}$ which has only the trivial solution $(0,0,\ldots,0)$. \cite[Lemma 2.8]{CKW} shows
(over the complex field) a simple necessary condition on $\mathcal{A}$ to be a regular system, namely that $m!$ divides the product of the integers in $\mathcal{A}$. In this context, the major issue to find easily expressible sufficient conditions turned out be surprisingly difficult. For the smallest case $m=3$, the conjecture is that the above necessary condition is also sufficient, that is, $\mathcal{A}=\{a,b,c\}$ is associated to a regular system whenever $abc\equiv 0 \pmod 6$. Evidences for this conjecture were  given in \cite{CKW} and later in \cite{chen}. For further developments related to regular sequences, see \cite{DZ,FS,GGW,KM,MSW}.

In positive characteristic, an algebraic closure of the prime field $\mathbb{F}_p$ with $p\ge 7$ is chosen for $\mathbb{K}$ and  two extremal cases are investigated further, namely $m=5$ and $m=p-1$. In the former case, $V$ can be viewed as the characteristic $p$ version of the Bring curve, and for infinitely many primes $p$ we prove that $V$ is an $\mathbb{F}_{p^2}$-maximal curve, that is the number of points of $V$ defined over $\mathbb{F}_{p^2}$ attains the Hasse-Weil upper bound $p^2+1+2\mathfrak{g}p=p^2+1+8p$. We also show that the smallest such primes are $29,59,149,239,839$. Our result gives a new contribution to the study of $\mathbb{F}_{p^2}$-maximal curves of small genera, initiated by Serre in 1985 and continued until nowadays; see \cite{bgkm,gklm,gkm,ser}. For $m=p-1$, $V$ is never $\mathbb{F}_{p^2}$-maximal instead. In fact, the number of points of $V$ does not increase passing from $\mathbb{F}_p$ to $\mathbb{F}_{p^2}$ so that $V$ has as many as $(p-2)!$ points defined over $\mathbb{F}_{p^2}$ whereas its genus is equal to $\frac{1}{4}(((p-3)(p-4)-4)(p-3)!)+1$; see Lemmas \ref{lem1210}, \ref{lem11octE}, \ref{lem12oct}, and Theorem \ref{nopuntiFp2}. A more general question for $m=p-1$ is to determine the length of the set $V(\mathbb{F}_{p^i})$ consisting of all points of $V$ in $\PG(p-2,\mathbb{F}_{p^i})$. For $i=1$, Lemma \ref{lem1210} shows $V(\mathbb{F}_{p})=(p-2)!$. Unfortunately, the elementary argument used in the proof of Lemma \ref{lem1210},  seems far away from being sufficient to tackle the general case. Even in the second particular case $i=2$, settled in Theorem \ref{nopuntiFp2}, the proof relies on the St\"ohr-Voloch bound and on the structure of the $G$-orbits on the points of $V$ determined in Section \ref{secautgroup}. On the other hand, Theorem \ref{the181021} yields $|V(\mathbb{F}_{p^d})|> |V(\mathbb{F}_{p})|$ where $d$ stands for the smallest integer such that $(p-2)\mid (p^d-1)$. The proof uses methods from Galois theory. It remains unsolved the problem whether $|V(\mathbb{F}_{p^i})|> |V(\mathbb{F}_{p})|$ holds for some intermediate value $i$.

An important question on the geometry of curves is to compute the possible intersection multiplicities $I(P,V\cap \Pi)$ between $V$ and hyperplanes $\Pi$ through a point $P$ generally chosen in $V$, also called orders. In the classical case (zero characteristic), the orders are precisely the non negative integers smaller than the dimension of the space where the curve is embedded, but this may fail in positive characteristic. If this occurs, then the curve is called non-classical. For plane curves, non-classicality means that all non-singular points of the curve are flexes.
In our case, Lemma \ref{lem17ag21} shows the existence of a hyperplane $\pi$ such that $I(P,V\cap \Pi)$ is at least $p$. Since $V$ is embedded in $\textrm{PG}(p-3,p)$, this implies that $V$ is non-classical; see Theorem \ref{the17ag21}. In Section \ref{secc1}, the intersection multiplicities $I(P,V\cap \Pi)$ for $P\in V(\mathbb{F}_p)$ are investigated. By Proposition \ref{pro211021}, such intersection multiplicities are $1,2,3$. This, together with Lemma \ref{lem17ag21}, yields that $1,2,3,p$ are orders of $V$. 

Another concept of non-classicality, due to St\"ohr and Voloch \cite{sv}, arose from their studies on the maximum number $N_{p^i}=N_{p^i}(\mathfrak{g},r,d)$ of points of an irreducible  curve of a given genus $\mathfrak{g}$, embedded in $\textrm{PG}(r,\mathbb{F}_{p^i})$  as a curve of degree $d$ can have. The St\"ohr-Voloch  upper bound on $N_{p^i}=N_{p^i}(\mathfrak{g},r,d)$ is known to be a deep result; in particular it may be used to give a proof for the  Hasse-Weil upper bound. Also, the St\"ohr-Voloch bound shows that the curves with large  $N_{p^i}=N_{p^i}(\mathfrak{g},r,d)$ have the property that the osculating hyperplane to the curve at a generically chosen point on it also passes through the Frobenius image of the point. A curve with this purely geometric property can only exist in positive characteristic, and if this occurs the curve is called Frobenius non-classical. Apart from the trivial case $p^i=2$, Frobenius non-classical curves are also non-classical. Theorem \ref{the191021} shows that $V$ is Frobenius non-classical. However this does not hold true for the quotient curves of $V$ with respect to the stabilizer of the coordinates. For instance, for $m=p-1$, $V_{m-3}$ has degree $p-3<p$ and hence it is a classical (and Frobenius classical) plane curve.

Over a finite field, the number of solutions of a system of diagonal equations has been the subject of many papers where the authors mainly rely on character sums and the distributions of their values. In the present paper we are not moving in this direction.
We point out a connection between (\ref{sy}) over $\mathbb{F}_p$ and R\'edei's work related with the famous Minkowski conjecture, originally proven by Haj\'os in 1941. R\'edei proved that Minkowski conjecture holds if the following claim is true: if an elementary abelian group of order $p^2$ is factored as the product of two sets of length $p$, both containing the identity element, then at least one of the factors is a subgroup; see \cite{redei}. R\'edei and later on Wang, Panico and Szab\'o \cite{szabo} showed this claim is true if each solution $[\xi_1,\ldots,\xi_p]$ over $\mathbb{F}_p$  of the system of diagonal equations
\begin{equation}
\label{syA}\left\{
\begin{array}{llll}
X_1 + X_2 + \ldots + X_p=0;\\
X_1^2  +  X_2^2 + \ldots + X_p^2=0;\\
\cdots\cdots\\
\cdots\cdots\\
X_1^{(p-1)/2}  +  X_2^{(p-1)/2}  +  \ldots+X_p^{(p-1)/2}=0.
\end{array}
\right.
\end{equation}
 has either equal components or $[\xi_1,\ldots,\xi_p]$ is a permutation of the elements of $\mathbb{F}_p$. That such particular $p$-tuples are exactly the solutions over $\mathbb{F}_p$ of system (\ref{syA}) was first shown by R\'edei himself in \cite{redei}. The above quoted paper by  Wang, Panico and Szab\'o also contains a proof.  In Section \ref{secredei},  we give a geometric
interpretation of their results in terms of the variety $V$.

\section{Basic definitions and notation}
\label{back} In this paper, $p\ge 7$ stands for a prime and $\mathbb{K}$ for an algebraically closed field of characteristic zero or $p$. Also, we let $m\geq 6$, and assume $m\leq p-1$ when $\mathbb{K}$ has characteristic $p$. For a projective, geometrically irreducible, non-singular curve $\cX$, the automorphism group of $\cX$ fixing $\mathbb{K}$ element-wise is denoted by $\aut(\cX)$. If $\cX$ has genus $\gg\ge 2$ then $\aut(\cX)$ is finite and Henn's classification theorem states that $|\aut(\cX)|<8 \gg^3$ with exactly two exceptions, namely the Hermitian curve and the Roquette curve. In the exceptional cases, no subgroup of $\aut(\cX)$ is isomorphic to $\rm{Alt}_m$; see \cite[Theorem 11.127]{HKT}.

\subsection{Action of $\rm{Sym}_m$ as a projectivity group of $PG(m-2,\mathbb{K})$}
Henceforth, ${\bf{V}}_{m}$ denotes the $m$-dimensional $\mathbb{K}$-vector space  with
basis $\{X_1,X_2,\ldots,X_{m}\}$, and $\textrm{PG}(m-1,\mathbb{K})$  the $(m-1)$-dimensional projective space over $\mathbb{K}$ arising from  ${\bf{V}}_{m}$.

The group $G={\rm{Sym}}_{m}$, viewed as the symmetric group on $\{X_1,\ldots, X_{m}\}$, has a natural faithful representation in ${\bf{V}}_{m}$, where  $g\in G$ takes the vector
${\bf{v}}=\sum_{i=1}^{m}\lambda_iX_i$ to the vector
${\bf{v}}'=\sum_{i=1}^{m}\lambda_ig(X_i)$. With ${\bf{v}}'=g({\bf{v}})$, $g$ becomes a linear transformation of ${\bf{V}}_{m}$ and $G$ is viewed as a subgroup of $\textrm{GL}(m,\mathbb{K})$. Since no non-trivial element of
$G$ preserves each $1$-dimensional subspace of ${\bf{V}}_{m}$, $G$ acts faithfully on
$\textrm{PG}(m-1,\mathbb{K})$ and $G$ can also be regarded as a subgroup of the projective linear group $\textrm{PGL}(m,\mathbb{K})$ of $\textrm{PG}(m-1,\mathbb{K})$.
For $1\le i < j \le m$, the transposition $\sigma=(X_iX_j)$, as a linear transformation of ${\bf{V}}_{m}$, is associated with the $(0,1)$-matrix $(g_{l,n})_{l,n}$ whose $1$ entries are $g_{l,n}$ with $l=n$ and $g_{ij},
g_{ji}$. As a projectivity of $\textrm{PG}(m-1,\mathbb{K})$, $\sigma$ is the involutory homology whose center is the point
$C=(0:\cdots: 1:\cdots:-1:\cdots:0)$, where $1$ and $-1$ are in $i$-th and $j$-th positions, respectively,
and whose axis is the hyperplane $\Pi$ of equation $X_i=X_j$. Therefore, the fixed points of $\sigma$ are those of $\Pi$ and $C$, whereas the fixed hyperplanes of $\sigma$ are those through $C$
and $\Pi$. Since any two transpositions with no common fixed point commute, the center of either lies on the axis of the other.
The $m$-cycle $\sigma=(X_{m}X_{m-1}\ldots X_1)$, as a linear transformation of ${\bf{V}}_{m}$, is associated with the $(0,1)$-matrix $(g_{l,n})_{l,n}$ whose $1$ entries are $g_{l,l+1}$ for  $l=1,\ldots,m-1$ and $g_{m,1}$. As a
projectivity, $\sigma$ fixes the point $P_\omega=(\omega:\omega^2:\cdots:\omega^{m}=1)$, where $\omega$ is any element whose order divides $m$. Similarly, the $(m-1)$-cycle $\sigma=(X_{m-1}X_{m-2}\cdots X_1)$, as a
linear transformation of ${\bf{V}}_{m}$, is associated with the $(0,1)$-matrix $(g_{l,n})_{l,n}$ whose $1$ entries are $g_{l,l+1}$ for $l=1,\ldots,m-2$ and $g_{m-1,1},g_{m,m}$. As a projectivity, $\sigma$ fixes the point
$P_\varepsilon=(\varepsilon:\varepsilon^2:\cdots:\varepsilon^{m-1}=1:0)$ for any element $\varepsilon$ of $\mathbb{K}$ whose order divides $m-1$.
\subsection{Projections}
We also recall how the projection from a linear subspace $\Sigma$ of dimension $d$ to a disjoint linear subspace $\Sigma'$ is performed.



If $\cC$ is an irreducible, non-singular algebraic curve embedded in $\textrm{PG}(m-1,\mathbb{K})$, which is disjoint from $\Sigma$, then the projection
determines a regular mapping $\pi: \cC \rightarrow \textrm{PG}(m-2-d,\mathbb{K})$.  The geometric interpretation of $\pi$ is straightforward. Take any ($m-2-d$)-dimensional subspace $\Sigma'$ of $\textrm{PG}(m-1,\mathbb{K})$ disjoint from
$\Sigma$. Then
through any point of $P\in \cC$ there is a unique ($d+1$)-dimensional subspace $\Lambda_P$ containing $\Sigma$ and this subspace meets $\Sigma'$ in a unique point, the image $\pi(P)$ of $P$ projected from $\Sigma$.
Moreover, $\cC$ is projected into an irreducible, possibly singular, curve $\cF$ embedded in $\Sigma'\cong \textrm{PG}(m-2-d,\mathbb{K})$.\\ In the particular case where $\Sigma$ and $\Sigma^{\prime}$ are the subspaces defined as the intersections of the hyperplanes of equations $X_{d+2},\ldots,X_{m}=0$ and $X_{1},\ldots,X_{d+1}=0$, respectively, then the projection from the vertex $\Sigma$ to $\Sigma^{\prime}$ maps a point $P=(x_1:\ldots:x_{m})\in \textrm{PG}(m-1,\mathbb{K})\setminus \Sigma$  to the point $P'=(0:\ldots:0:x_{d+2}:\ldots:x_{m})$.\\
Let $H$ be the subgroup of $\textrm{PGL}(m-1,\mathbb{K})$ which preserves $\cC$ and fixes each $d+1$-dimensional subspace through $\Sigma$. The points of the quotient curve $\bar{\cC}=\cC/H$ can be viewed  as $H$-orbits of the points of $\cC$. Therefore every point $\bar{P}\in\bar{C}$ is identified by a unique $H$-orbit, say $\overline{H_P}$.
Moreover, such an $H$-orbit $\overline{H_P}$ is contained in a unique ($d+1$)-dimensional subspace through $\Sigma$, which also meets $\cC$ in a unique point.  Therefore, the notation $\Lambda_P$ for that subspace
passing through the point $P\in\cC$ is meaningful. Thus the sequence $\bar{P}=\overline{H_P}\rightarrow \Lambda_P \rightarrow \pi(P)$ is well defined. Actually it is a
surjective homomorphism from $\bar{\cC}$ to $\cF$. It is bijective if and only if $|H|=|\Lambda_P\cap\cC|$ for all but finitely many points $P\in \cC$.
In this case $\Sigma$ is called a $d$-dimensional outer \emph{Galois subspace}, and for $d=0$ an outer \emph{Galois}-point; see \cite{KLT}.
If $N$ is a subgroup of the normalizer of $H$ in the $\mathbb{K}$-automorphism group of $\cC$ then the quotient group $N/H$ is a subgroup of the $\mathbb{K}$-automorphism group of $\bar{\cC}$.
The genera $\mathfrak{g}(\cC)$ and $\mathfrak{g}(\bar{\cC})$ of the curves   $\cC$ and $\bar{\cC}$ are linked by the Hurwitz genus formula. In particular, if $H$ is tame, that is, the characteristic of $\mathbb{K}$ is either zero, or
equal $p$ with $p$ prime to the order $\ell=|H|$ of $H$, then
\begin{equation}
\label{hufo}
2\mathfrak{g}(\cC)-2=|H| (2\mathfrak{g}(\bar{\cC})-2)+\sum_{i=1}^m(\ell-\ell_i),
\end{equation}
where $\ell_1,\ldots,\ell_r$ denote the lengths of the short-orbits of $H$ on $\cC$.

The equations in (\ref{sy}) define a (possibly reducible and singular) projective variety $V$ of $\textrm{PG}(m-1,\mathbb{K})$ so that the points of $V$ are the nontrivial solutions of (\ref{sy}) up to a non-zero scalar.
Actually, $V$ is contained in the hyperplane of $\textrm{PG}(m-1,\mathbb{K})$ of equation $X_1+X_2+\ldots X_{m}=0$. Therefore, $V$ is a projective variety embedded in $\textrm{PG}(m-2,\mathbb{K})$. Moreover, $G={\rm{Sym}}_{m}$
preserves $V$ and no nontrivial element of $G$ fixes $V$ pointwise. Therefore, $G$ is a subgroup of the $\mathbb{K}$-automorphism group of $V$.

\subsection{B\'ezout's theorem}
The higher dimensional generalization of B\'ezout's theorem about the number of common points of $m-1$ hypersurfaces $\cH_1,\ldots,\cH_{m-1}$ of $\textrm{PG}(m-1,\mathbb{F}_q)$ states that either that number is infinite or does
not exceed the product $\deg(\cH_1)\cdot\ldots \cdot\deg(\cH_{m-1})$.
For a discussion on B\'ezout's theorem and its generalization; see \cite{Vog}.

\subsection{Background on algebraic curves}
We report some background from \cite[Chapter 7]{HKT}; see also \cite{sv}.
Let $\Gamma$ be a projective, absolutely irreducible, not necessarily non-singular curve, embedded in a projective space ${\rm{PG}}(r,\mathbb{K})$.
For a non-singular model $\cX$ of $\Gamma$, let $\mathbb{K}(\cX)=\mathbb{K}(\Gamma)$ denote the function field of $\cX$. There exists a bijection between the points of $\cX$, and the places of $\mathbb{K}(\cX)$ and the branches of $\Gamma$.
For any point $P\in\Gamma$ (more precisely, for any branch of $\Gamma$ centered at $P$) the different intersection multiplicities of hyperplanes with the curve at $P$ are considered. There is only a finite number of these intersection multiplicities, the number being equal to $r+1$. There is a unique hyperplane, called osculating hyperplane, with the maximum intersection multiplicity. The hyperplanes cut out on $\Gamma$ a simple,
fixed point free, not-necessarily complete linear series $\Sigma$ of dimension $r$ and degree $n$, where $n$ is the degree of the curve $\Gamma$. An integer $j$ is a $(\Sigma,P)$-order if there is a hyperplane $H$ such that $I(P,H\cap \Gamma)=j$. Notice that, if $P$ is a singular point, then $P$ is intended as a branch of $\Gamma$ centered at $P$.  In the case that $\Sigma$ is the canonical series, it follows from the
Riemann--Roch theorem that $j$ is a $(\Sigma,P)$-order if and only if $j+1$ is a
Weierstrass gap. For any non-negative integer $i$, consider the set of all hyperplanes $H$
of $\PG(r,\mathbb{K})$ for which the intersection number is at least $i$. Such
hyperplanes correspond to the points of a subspace $\overline{\Pi}_i$ in
the dual space of $\PG(r,K)$. Then we have  the decreasing chain
$$
\PG(r,K) = \overline{\Pi}_0 \supset
\overline{\Pi}_1 \supset \overline{\Pi}_2 \supset \cdots.
$$

An integer $j$ is a $(\Sigma,P)$-order if and only if $\overline{\Pi}_j$ is not equal to the subsequent space in the chain.
In this case $\overline{\Pi}_{j+1}$ has
codimension 1 in $\overline{\Pi}_j$. Since $\deg\,\, \Sigma=n$, we have that $\overline{\Pi}_i$ is
empty as soon as $i> n$. The number of $(\Sigma,P)$-orders is exactly
$r+1$; they are $j_0(P),j_1(P),\ldots,j_r(P)$ in increasing order, and
$(j_0(P),j_1(P),\ldots,j_r(P))$ is the order-sequence of $\Gamma$ at $P$.

Here $j_0(P)=0$, and
$j_1(P)=1$ if and only if the branch $P$ is linear (in particular when $P$ is a non-singular point).

Consider the intersection $\Pi_i$ of hyperplanes $H$ of $\PG(r,K)$,
for which
$$
I(P,H\cap\g)\geq j_{i + 1}.
$$ Then the flag
$
\Pi_0\subset \Pi_1\subset \cdots \subset \Pi_{r-1} \subset
\PG(r,K)
$
can be viewed as the algebraic analogue of the Frenet
frame in differential geometry.
 Notice that $\Pi_0$ is just $P$, the
centre of the branch $\g$, and $\Pi_1$ is the tangent line to the
branch $\g$ at $P$. 
Furthermore, $\Pi_{r-1}$ is the osculating hyperplane at $P$.

The order-sequence is the same for all but finitely many points of $\Gamma$, each such exceptional point is called a $\Sigma$-Weierstrass point of $\Gamma$.
The order-sequence at a generally chosen point of $\Gamma$ is the order sequence of $\Gamma$ and denoted by $(\varepsilon_0,\varepsilon_1,\ldots,\varepsilon_r)$. Here $j_i(P)$ is at least $\varepsilon_i$ for $0\le i \le r$ at any point of $\Gamma$.

Now let $\Gamma$ be defined over $\mathbb{F}_\ell$ and viewed as a curve over the algebraic closure $\mathbb{K}=\bar{\mathbb{F}}_\ell$. St\"ohr and Voloch \cite{sv} (see also \cite[Chapter 8]{HKT}) introduced a divisor with support containing all points of $P\in \Gamma$ for which the osculating hyperplane contains the Frobenius image $\Phi(P)$ of $P$. Since every $\mathbb{F}_\ell$-rational point has this property, an upper bound on the number of $\mathbb{F}_\ell$-rational points is obtained, unless all (but finitely many) osculating hyperplanes have that property. In such an exceptional case, the curve is called Frobenius non-classical. Curves with many rational points are often Frobenius non classical (and in particular, non-classical for $q\neq 2$). Actually, St\"ohr and Voloch were able to give
a bound on the number of $\mathbb{F}_\ell$-rational points for any (Frobenius classical, or non-classical) curve $\Gamma$. There exists a sequence of increasing non-negative integers
$\nu_0,\ldots,\nu_{r-1}$ with $\nu_0\geq 0$ such that
$
\det(W_{\gz}^{\nu_0,\ldots,\nu_{r-1}}(x_0,\ldots,x_r)) \neq 0.
$
In fact, if the
$\nu_i$ are chosen minimally in lexicographical order$,$ then there
exists an integer $I$ with $0 < I \leq r$ such that
$$
\nu_i = \left\{ \begin{array}{ll}
\varepsilon_i & \quad \mbox{for $i < I$},\\
\varepsilon_{i+1} & \quad \mbox{for $i \geq I$}.
\end{array}\right.
$$

The St\"ohr-Voloch divisor of $\Gamma$ is
$$ S = \Div(W_{\gz}^{\nu_0,\ldots,\nu_{r-1}}(x_0,\ldots,x_r)) +
(\nu_1 + \cdots + \nu_{r-1})\,\Div(d\gz) + (q + r)E,
$$
where $E$ and $e_P$ are defined as before for the ramification divisor.
The sequence $(\nu_0,\nu_1,\ldots,\nu_{r-1})$ is the Frobenius order-sequence, and $\Gamma$ is Frobenius-classical if $\nu_i=i$ for $0\le i \le r$. Also,
$\deg\, S = (\nu_1 + \cdots + \nu_{r-1})(2\mathfrak{g} - 2) + (q+r)n$, where $n=\deg(\Gamma).$ From this the St\"ohr-Voloch bound follows:
$$|\Gamma(\mathbb{F}_\ell)|\le \frac{1}{r}\Big( (\nu_1 + \cdots+ \nu_{r-1})(2\mathfrak{g}-2) + (\ell +r)n\Big),$$
which is a deep result on the number of $\mathbb{F}_{\ell}$-rational points of $\mathbb{F}_\ell$-rational curves. For instance, the Hasse-Weil upper bound was re-proven in \cite{sv}; see also \cite[Chapter 9]{HKT}.

\subsection{Preliminary results from projective geometry and linear groups in positive characteristic}
 In the $r$-dimensional projective space $PG(r,\mathbb{K})$ with $r\ge 3$ and defined over $\mathbb{K}$, a (non-trivial) projectivity $u\in PGL(r+1,\mathbb{K})$ is a {\emph{perspectivity}} if $u$ fixes all points of a hyperplane called the axis, and preserve all lines through a point, called the center. From now on in this subsection, we assume $\mathbb{K}$ to be of positive characteristic $p\ge 7$. Then, $u$ is either an \emph{elation} or an \emph{homology} according as the axis contains the center or does not. Since $p\ne 2$, no elation preserves a non-singular quadric of $PG(r,\mathbb{K})$.
 
 Let $\mathbb{F}_q$ be a finite subfield of $\mathbb{K}$ of order $q$. 
 We recall some basic facts about projective orthogonal groups over $\mathbb{F}_q$.  If $r$ is even then there is up to conjugacy in $PGL(r+1,q)$ just one projective orthogonal group $P\Omega_{r+1}(q)$ whereas for $r$ odd there are two, $P\Omega^+(r+1,q)$ and $P\Omega^-(r+1,q)$.

A projective orthogonal group has exactly two classes of conjugate involutory homologies, each of them generates an index $2$ subgroup $U$ containing no elations. Furthermore,
$$|U|=
\begin{cases}
\ha q^{n^2}(q^2-1)(q^4-1)\cdots (q^{2n}-1)\,\, {\mbox{for $r=2n$,}}\\
\ha q^{n(n-1}(q^2-1)(q^4-1)\cdots (q^{2n}-2)(q^n-1)\,\, {\mbox{for $r=2n-1$ and $U<P\Omega^+(r+1,q)$,}}\\
\ha q^{n(n-1}(q^2-1)(q^4-1)\cdots (q^{2n}-2)(q^n+1)\,\, {\mbox{for $r=2n-1$ and $U<P\Omega^-(r+1),q)$.}}
\end{cases}
$$
 The following result is a consequence of the classification theorem of Zalesski{\rm{$\breve{i}$}} and Sere$\hat{z}$kin \cite{ZS}; see also \cite{AW}.
\begin{result}
\label{zales} Let $N\le PGL(r+1,\mathbb{K})$  be a finite irreducible subgroup of the projective orthogonal group $P\Omega(r+1,\mathbb{K})$ for $r$ even, and of $P\Omega^+(r+1,\mathbb{K})$ or $P\Omega^-(r+1,\mathbb{K})$ for $r$ odd. Assume that $N$  is generated by a conjugacy class of involutory homologies. If $r\ge 8, p>r$ and $p$ divides $|N|$
then $N$ is conjugate in $PGL(r+1,\mathbb{K})$ either to a projective orthogonal group over a finite subfield of $\mathbb{K}$, or to one of its two subgroups $U$ generated by a class of conjugate involutory homologies, unless $N\cong \rm{Sym}_{r+1}$ or $N\cong \Sym_p$ with $p=r+3$.
\end{result}
As a corollary, we have the following result.
\begin{result}
\label{zalesA} Let $N\le PGL(m-1,\mathbb{K})$  be a finite irreducible subgroup of the projective orthogonal group $P\Omega(m-1,\mathbb{K})$ for $m$ odd, and of $P\Omega^+(m-1,\mathbb{K})$ or $P\Omega^-(m-1,\mathbb{K})$ for $m$ even. Assume that $N$  is generated by involutory homologies. If $m\ge 8, p\ge m+1$ and $p$ divides $|N|$
then either
\begin{equation}
\label{eq270624}
|N|>8(\ha((m-2)(m-3)-4)(m-2)!)+2)^3,
\end{equation}
or $p=m+1$ and $N$ is isomorphic to $\Sym_p$.
\end{result}

We also report the following classification theorem from Guralnick; see \cite{gur}.
\begin{result}
\label{gu} Let $\Gamma$ be an almost simple transitive permutation group of prime degree $p$. Then, up to an isomorphism, $T\le \Gamma\le \aut(T)$ for some non-abelian simple group $T$ where $T$ is isomorphic to one of the following groups:
\begin{itemize}
\item[(i)] $\rm{Alt}_p$;
\item[(ii)]$PSL(r,d^h)$ with $p=(d^r-1)/d-1)$, $r$ prime, and $PSL(r,d)\le \Gamma \le P\Gamma L(r,d)$;
\item[(iii)] $PSL(2,11)$ for $p = 11$;
\item[(iv)] $M_{11}$ for $p = 11$, or $M_{23}$ for $p = 23$.
\end{itemize}
\end{result}
\begin{result}
\label{shar} For an odd prime power $d$, let $H$ be a cyclic subgroup of $P\Gamma L(r,d)$ fixing a point $P$. Then $H$ is not sharply transitive on the remaining $d^{r-1}+\ldots+d$ points of $PG(r,d)$.
\end{result}
\begin{proof} Look at the set $\cL$ consisting of all lines through $P$. The subgroup $H$ is a collineation group of $PG(r,d^h)$ and induces on $\cL$ a transitive permutation group. The point $P$ is the center of each non-trivial collineation in the kernel $K$ of the permutation representation of $H$ on $\cL$. Thus, $\bar{H}=H/K$ is a collineation group of the quotient projective space $PG(r-1,d)$ of $PG(r,d)$ with respect to $\cL$. Clearly, $\bar{H}$ acts transitively on the set of all points of $PG(r-1,d)$. Since $H$ is cyclic, this action is regular. Hence, $|\bar{H}|$ is equal to the number of points of $PG(r-1,d)$. From $d^{r-1}+\ldots+d=|H|=|\bar{H}||K|=(d^{r-2}+\ldots+d+1)|K|$ whence $|K|=d$. Since $d$ is a power of the characteristic of the coordinatizing field of $PG(r,d)$, and $K$ has order $d$, it turns out that $K$ is an elation group. Let $\pi$ be the axis of $K$. Since $K$ centralizes $H$, it turns out that $\pi$ is preserved $H$. But then $H$ is not transitive on the set of points of $PG(r,d)$ different from $P$.
\end{proof}
\subsection{Power sums and elementary symmetric polynomials}
\label{powersum}
In $\mathbb{K}[X_1,\ldots,X_m]$, let
$$
p_k(X_1, \ldots, X_m) = \sum_{i=1}^{m} X_i^k, \quad k\ge 1
$$
and let $\sigma_k(X_1, \ldots, X_n)$ be the $k^{th}$ elementary symmetric polynomial for $k\ge 0$.  The Newton identities are
$$
k \sigma_k(X_1, \ldots, X_m) = \sum_{i=1}^{k} (-1)^{i-1} \sigma_{k-i}(X_1, \ldots, X_m) p_i(X_1, \ldots, X_m), \quad m\ge k\ge 1.
$$
whence
\begin{equation*}
p_{k}(X_{1},\ldots ,X_{m}) = (-1)^{k-1}k \sigma_{k}(X_{1},\ldots ,X_{m}) + \sum_{i=1}^{k-1} (-1)^{k-1+i} \sigma_{k-i}(X_{1},\ldots ,X_{m}) p_{i}(X_{1},\ldots ,X_{n}),\quad 1\le k\le m.
\end{equation*}
Therefore, (\ref{sy}) is equivalent to system
\begin{equation}
\label{sysA}
\sigma_1(X_1,\ldots,X_m)=\sigma_2(X_1,\ldots,X_m)=\cdots =\sigma_{m-2}(X_1,\ldots,X_m)=0.
\end{equation}

\section{Some particular solutions of system  (\ref{sy}) of diagonal equations}

\begin{lemma}
\label{lem6pct} Let $\varepsilon\in\mathbb{K}$ be a $(m-1)$-th primitive root of unity. Then $(\varepsilon,\varepsilon^2,\ldots,\varepsilon^{m-1}=1,0)$ is a solution of system \eqref{sy}.
\end{lemma}
\begin{proof}
Take an integer $i$ such that $1\le i \le m-2$, and let $\theta=\varepsilon ^i$. Then $\theta$ is a ($m-1$)-th root of unity (non necessarily primitive). Furthermore, $\theta\neq 1$ as $\varepsilon$ is a primitive
($m-1$)-th root of unity.  The $i$-th equation in (\ref{sy}) is satisfied by $(\varepsilon,\varepsilon^2,\ldots,\varepsilon^{m-1},0)$ if and only if $\sum_{k=1}^{m-1} \theta^k=0$.  This sum equals
$(\theta^{m}-1)/(\theta-1)-1$, and hence it is zero.
\end{proof}
The same argument also proves the following result.
\begin{lemma}\label{lem1210}
Let $\omega\in \mathbb{K}$ be a $m$-th primitive root of unity. Then $(\omega,\omega^2,\ldots,\omega^{m}=1)$ is a solution of system \eqref{sy}.
\end{lemma}

\begin{lemma}
\label{lem11octA} System (\ref{sy}) has no nontrivial solution $(x_1,x_2,\ldots,x_{m-1},x_{m})$ for $x_{m-1}=x_{m}=0.$
\end{lemma}
\begin{proof} For a solution ${\bf{x}}=(x_1,x_2,\ldots,x_{m-1},x_{m})$ of system (\ref{sy}),  let $y_1,\ldots,y_k$ be the pairwise distinct non-zero coordinates of ${\bf{x}}$. For $1\le j \le k$, the multiplicity of
$y_j$ is defined to be the number $n_j$ of the coordinates of ${\bf{x}}$ which are equal to $y_j$. Since $x_{m-1}=x_{m}=0$ we have $k\le m-2$.  The first $k$ equations of (\ref{sy}) read
 \begin{equation}
\label{sy11}\left\{
\begin{array}{llll}
n_1y_1+n_2y_2+\ldots + n_ky_k =0;\\
n_1y_1^2+n_2y_2^2+\ldots+n_ky_k^2=0;\\
\cdots\cdots\\
\cdots\cdots\\
n_1y_1^k+n_2y_2^k+ \ldots+n_ky_k^k=0.
\end{array}
\right.
\end{equation}
Then $(y_1,y_2,\ldots,y_k)$ is a nontrivial solution of the linear system
\begin{equation}
\label{sy11U}\left\{
\begin{array}{llll}
n_1X_1 + n_2X_2 + \ldots + n_kX_k=0;\\
n_1y_1X_1  + n_2y_2X_2+\ldots +n_ky_kX_k=0;\\
\cdots\cdots\\
\cdots\cdots\\
n_1y_1^{k-1}X_1 + n_2y_2^{k-1}X_2+ \ldots+n_ky_k^{k-1}X_k=0.
\end{array}
\right.
\end{equation}
whose determinant is equal to the product of $\prod_{i=1}^k n_i$ by the $k\times k$ Vandermonde determinant $\Delta=\prod_{1\le i <j \le k} (y_i-y_j)$. Since either $p=0$ or, if $p>0$ then $n_i<p$ holds, it turns
out that $\Delta=0$ and hence $y_i=y_j$ for some $i\ne j$. But this contradicts the definition of $y_1,\ldots,y_k$.
\end{proof}
\begin{lemma}
\label{lem11octB} No nontrivial solution $(x_1,x_2,\ldots,x_{m-1},x_{m})$ of system (\ref{sy}) has either
\begin{itemize}
\item[\rm(i)] four coordinates $x_{i_1},x_{i_2}, x_{i_3},x_{i_4}$ such that $x_{i_1}=x_{i_2}$ and $x_{i_3}=x_{i_4}$ with pairwise distinct indices $i_1,i_2,i_3,i_4$, or
\item[\rm(ii)] three coordinates $x_{i_1},x_{i_2}, x_{i_3}$ such that $x_{i_1}=x_{i_2}=x_{i_3}$ with pairwise distinct indices $i_1,i_2,i_3$, or
\item[\rm(iii)] three coordinates $x_{i_1},x_{i_2}, x_{i_3}$ such that $x_{i_1}=x_{i_2}$ and $x_{i_3}=0$ with two distinct indices $i_1,i_2$.
\end{itemize}
\end{lemma}
\begin{proof} We use the argument in the proof of Lemma \ref{lem11octA}. For a non-trivial solution ${\bf{x}}=(x_1,x_2,\ldots,x_{m-1},x_{m})$ of system (\ref{sy}), let $y_1,\ldots,y_k$ be the pairwise distinct non-zero
coordinates of ${\bf{x}}$ with multiplicities $n_j$ for $j=1,\ldots, k$. If ${\bf{x}}$ is a counterexample to Lemma \ref{lem11octB} then $k\le m-2$. But the proof of Lemma \ref{lem11octA} shows that this is actually
impossible.
\end{proof}
\begin{lemma}
\label{lem14oct} Up to a non-zero scalar, system (\ref{sy}) has finitely many solutions $(x_1,x_2,\ldots,x_{m-2},x_{m-1},0)$.
\end{lemma}
\begin{proof} We again use the idea from the proof of Lemma \ref{lem11octA}. For a non-trivial solution ${\bf{x}}=(x_1,x_2,\ldots,x_{m-1},0)$ of system (\ref{sy}), let $y_1,\ldots,y_k$ be the pairwise distinct non-zero
coordinates of ${\bf{x}}$ with multiplicities $n_j$ for $j=1,\ldots, k$. Furthermore, one of them can be assumed to be equal to $1$.  Up to a relabelling of indices, $n_k$ counts the coordinates equal to $1$. Then $(y_1,\ldots,y_{k-1})$ is a solution of
\begin{equation}
\label{sy111}\left\{
\begin{array}{llll}
n_1X_1 + n_2X_2 + \ldots + n_{k-1}X_{k-1}=-n_k;\\
n_1y_1X_1  + n_2y_2X_2+\ldots +n_{k-1}y_{k-1}X_{k-1}=-n_k;\\
\cdots\cdots\\
\cdots\cdots\\
n_1y_1^{k-2}X_1 + n_2y_2^{k-2}X_2+ \ldots+n_{k-1}y_{k-1}^{k-2}X_{k-1}=-n_k.
\end{array}
\right.
\end{equation}
Assume on the contrary that system (\ref{sy111}) has infinitely many solutions. Since $n_k\neq 0$, the determinant of (\ref{sy111}) vanishes. On the other hand, up to a non-zero constant, it is a Vandermonde determinant
with parameters $(y_1,\ldots,y_{k-1})$. Therefore,  $y_i=y_j$ must occur for some $i\ne j$ contradicting the definition of $y_1,\ldots,y_k$.
\end{proof}
\section{The variety defined by system \eqref{sy}}
\label{secvar}
We keep upon with the notation introduced in Section \ref{back}. In particular, $V$ stands for the projective algebraic variety of ${\rm{PG}}(m-1,\mathbb{K})$ defined by the equations of system (\ref{sy}).

Our goal is to prove that $V$ is an irreducible non-singular curve. This requires some technical lemmas involving both the geometry of $V$ and the transpositions of $G$. We begin by stating and proving them.
\begin{lemma}
\label{lem1D} $V$ has dimension $1$.
\end{lemma}
\begin{proof}
From \cite[Corollary 5]{sha} applied to $r=m-2$ and $n=m-1$, we have $\dim(V)\ge 1$. On the other hand, we show that there exists a projective subspace of codimension $2$ which is disjoint from $V$. By  \cite[Corollary
4]{sha}, this will yield $\dim(V)\le m-1-(m-1-2)-1\le 1$. Lemma \ref{lem11octA} ensures that a good choice for such a subspace of codimension 2 is the intersection $\Lambda$ of the hyperplanes of equation $X_{m}=0$ and
$X_{m-1}=0$.
\end{proof}
\begin{lemma}
\label{lem1} $V$ is non-singular.
\end{lemma}
\begin{proof}  The Jacobian matrix of $V$ is
$$\nabla(V)=\frac{\partial(f_1,\ldots,f_{m-2})}{\partial(X_1,\ldots,X_{m})}=
\left(
  \begin{array}{llll}\vspace{0.1cm}
    \frac{\partial f_1}{\partial X_1} & \cdots & \frac{\partial f_1}{\partial X_{m}} \\
    \frac{\partial f_2}{\partial X_1} & \cdots & \frac{\partial f_2}{\partial X_{m}} \\
    \cdots & \cdots & \cdots  \\
     \frac{\partial f_{m-2}}{\partial X_1} & \cdots & \frac{\partial f_{m-2}}{\partial X_{m}} \\
  \end{array}
\right)=
\left(
  \begin{array}{llll}\vspace{0.1cm}
    1 & \cdots & 1  \\
    2X_1 & \cdots & 2X_{m} \\
    \cdots & \cdots & \cdots  \\
     (m-2)X_1^{m-3} & \cdots & (m-2)X_{m}^{m-3} \\
  \end{array}
\right).
$$
Up to the non-zero factor $(m-2)!$, the determinants of maximum order $m-2$ are Vandermonde determinants. 
Therefore, for some point $P=(x_1:\ldots:x_{m})\in V$, $\nabla(V)$ evaluated at $P$ has rank less than $m-2$ if and only if $(x_1,\ldots,x_{m})$ has one of the properties (i) and (ii).
But Lemma \ref{lem11octB} rules out these possibilities, and hence the point $P$ is non-singular.
\end{proof}
\begin{lemma}
\label{lem8oct} $\deg(V)=(m-2)!$.
\end{lemma}
\begin{proof}
Let $\Pi$ be the hyperplane with equation $X_{m-1}=X_{m}$. Since $V$ has dimension $1$, $\Pi$ intersects $V$ in at least one point $P=(x_1:x_2:\ldots:x_{m-2}:x_{m-1}:x_{m-1})$. By Lemma \ref{lem11octA}, $x_{m-1}\neq 0$ and so $P=(x_1/x_{m-1}:\ldots:x_{m-2}/x_{m-1}:1:1$). Also, from Lemma \ref{lem11octB} $x_1/x_{m-1},\ldots,x_{m-2}/x_{m-1}$ are pairwise distinct. Since any permutation of $x_1/x_{m-1},\ldots,x_{m-2}/x_{m-1}$ gives rise to a new point of $V$ on $P$, we obtain $|\Pi\cap V|=(m-2)!$ and hence $\deg(V)\ge (m-2)!$. On other hand, the higher dimensional generalization of B\'ezout's theorem yields $\deg(V)\le (m-2)!$ whence the claim follows.
\end{proof}

\begin{lemma}
\label{lem10giugno} Let $P=(\xi_1:\xi_2:\cdots:\xi_{m})$ be a point of $V$, and fix $r\ge 2$ indices  $1\le j_1<\cdots <j_r\le m$. Let $R=(x_1:x_2:\cdots:x_{m})$ be a point of $V$ such that $x_{j_1}=\xi_{j_1},\ldots, x_{j_r}=\xi_{j_r}$. Then there is a permutation $\rho$ on $\{1,\ldots,m\}$ with $\rho(j_i)=j_i$ for $i=1,\ldots,r$  such that $x_k=\xi_{\rho(k)}$ for $k=1,\ldots,m$.
\end{lemma}
\begin{proof} If there is a pair $\{i,\ell\}$ with $i\ne \ell$ such that $\xi_{j_i}=\xi_{j_\ell}$, let $\Pi$ be the hyperplane of equation $X_{j_1}=X_{j_\ell}$. Lemma \ref{lem11octB} shows that the $m-2$ coordinates of $P$ other than $\xi_{j_i}$ and $\xi_{j_\ell}$ are pairwise distinct. Therefore the permutations on the coordinates other than $X_{j_i}$ and $X_{j_\ell}$ give rise as many as $(m-2)!$ pairwise distinct points of $V$. By Lemma \ref{lem8oct} these are all points in $V\cap \Pi$. From this the claim follows. If no such pair $\{i,\ell\}$ exists, take for $\Pi$ the hyperplane of equation $\Pi: X_{j_1}=\xi_{j_1}X_{m}$. Then the above argument still works, and the claim holds true.
\end{proof}

\begin{lemma}
\label{lem10octAchar0} For a $m$-th primitive root of unity $\omega$, let $P_\omega=(\omega:\omega^2:\cdots:\omega^{m}=1)$ be a point of $V$. Then the stabilizer of $P_\omega$ in $G$ is a cyclic group of order
$m$, and it acts on $\{X_1,\ldots,X_{m}\}$ as a $m$-cycle. Moreover if $\mathcal{O}_\omega$ is the $G$-orbit of $P_\omega$, then $|\mathcal{O}_\omega|=(m-1)!$, and if $\Pi_\omega$ is the hyperplane $X_{m-1}=\omega
X_{m}$, then $\Pi_\omega\cap \mathcal{O}_\omega=V\cap \Pi_\omega$.
\end{lemma}
\begin{proof}
Let $u\in {\rm{PGL}}(m-1,\mathbb{K})$ be given by the $(0,1)$-matrix $\{g_{i,j}\}$ whose $1$ entries are $g_{i,i+1}$ for $i=1,2,\ldots, p-2$,  and $g_{m,1}$. Clearly, $u\in G$ and $P_\omega$ is fixed
by $u$. Also $u$ has order $m$ and acts on $\{X_1,\ldots,X_{m}\}$ as a $m$-cycle. Therefore, $|G_P|\geq m$. We show that equality holds. Let $\sigma$ be
permutation on $\{ 1, 2,\dots,m-1\}$. If $P_\omega=P_\omega^\sigma$, then $\omega^i=\omega^{\sigma(i)}$ for any $i\leq m-1$. Thus $\sigma(i)-i\equiv 0\pmod{m}$ whence $\sigma(i)=i$ follows. Therefore $|G_P|= m$. From
this
$|\mathcal{O}_\omega|=|G|/|G_P|=(m-1)!$ follows. Since $\mathcal{O}_\omega\subset \Pi_\omega$ and $|V\cap \Pi_\omega|\le \deg(V)$ by Lemma \ref{lem1},  the last claim follows from Lemma \ref{lem8oct}.
\end{proof}
\begin{lemma}
\label{lem7oct} For a $(m-1)$-th primitive root of unity $\varepsilon$, let $P_\varepsilon=(\varepsilon:\varepsilon^2:\ldots,\varepsilon^{m-1}=1:0)$. Then the stabilizer of $P_{\varepsilon}$ in $G$ is a cyclic
group of order $m-1$ which acts on $\{X_1,\ldots,X_{m-1}\}$ as a cycle.
\end{lemma}
\begin{proof} Let $h\in {\rm{PGL}}(m-1,\mathbb{K})$ be given by the $0,1$-matrix $\{g_{i,j}\}$ whose $1$ entries are $g_{i,i+1}$ for $i=1,2,\ldots, m-2$, $g_{m-1,1}$ and $g_{m,m}$. Clearly $h$ takes $P_{\varepsilon}$ to the
point $Q=(\varepsilon^2:\varepsilon^3:\ldots:1:\varepsilon:0)$. Actually, $Q=P_{\varepsilon}$ as the coordinates of $Q$ are proportional to those of $P_{\varepsilon}$. It remains to show that any $g\in G$ fixing
$P_{\varepsilon}$ is a power of $h$.
Let $g:\,(X_1:X_2:\ldots:X_{m-1}:0)\rightarrow (Y_1:Y_2:\ldots:Y_{m-1}:0)$ where $Y_1Y_2\cdots Y_{m-1}$ is a permutation $\pi$ of $X_1X_2\cdots X_{m-1}$. Since $g$ fixes $P_{\varepsilon}$,
there exists $(y_1:y_2:\ldots:y_{m-1})$ with $\pi(\varepsilon^i)=y_i$ for $i=1,\ldots,m-1$ such that $(\varepsilon,\varepsilon^2,\ldots,\varepsilon^{m-1}=1)$ and  $(y_1,y_2,\ldots,y_{m-1})$ are proportional. Then
$y_{m-1}\varepsilon^i=y_i$ for $i=1,\ldots, m-1$. Also, there exists $1\le j \le m-1$ such that $y_j=1$. Therefore $y_{m-1}=\varepsilon^{-j}$ whence $y_i=\varepsilon^{i-j}.$ Thus $\pi(\varepsilon^i)=\varepsilon^{i-j}$
whence $Y_i=X_{i+(m-1-j)}$ where the indices are taken modulo $m-1$. Hence $g=h^{m-1-j}$ which proves the first claim. Since $h$ fixes $X_{m}$, $G$ acts on $\{X_1,\ldots,X_{m-1}\}$ as a cycle.
\end{proof}

\begin{lemma}
\label{lem8octU} Let $\Pi_\infty$ be the hyperplane of equation $X_{m}=0$. Then $\Pi_\infty$ intersects $V$ transversally at each of their $(m-2)!$ common points. Moreover, $\mathcal{O}_\varepsilon=V\cap\Pi_\infty$ is
 the orbit of $P_\epsilon$ under the action of the stabilizer of $\Pi_\infty$ in $G$.
\end{lemma}
\begin{proof} Since $|G|=m!$, the subgroup of $G$ preserving $\Pi_\infty$ has order $(m-1)!$. Since $P_\varepsilon\in \Pi_\infty$, Lemma \ref{lem7oct} yields that $\Pi_\infty$ contains at least
$(m-1)!/(m-1)=(m-2)!$ pairwise distinct points of $V$.  On the other hand, $|V\cap \Pi_\infty |\le \deg(V)$ by Lemma \ref{lem1}. As $\deg(V)=(m-2)!$ by Lemma \ref{lem8oct}, the first claim follows. Since the
stabilizer of $\Pi_\infty$ in $G$ has order $(m-1)!$, the first claim together with Lemma \ref{lem7oct} prove the second claim.
\end{proof}
\begin{lemma}
\label{lem18oct} Let $g\in G$ be a nontrivial element fixing a point of $V$. If $g$ acts on ${\bf{V}}_{m}$ fixing $X_i$ and $X_j$ then $g$ is a transposition on $\{X_1,X_2,\ldots,X_{m}\}$ and an involutory homology of $PG(m-2,\mathbb{K})$.
Furthermore, the fixed points of $g$ in $V$ are as many as $(m-2)!$ and they are all the common points of $V$ with the hyperplane of equation $X_l=X_n$ where $X_n=g(X_l)$, and $l\ne n$.
\end{lemma}
\begin{proof} Up to a change of the indices, $g$ fixes $X_{m}$ and $X_{m-1}$. Let $P=(x_1,\ldots,x_{m-1},x_{m})\in V$ be a fixed point of $g$. From Lemma \ref{lem11octA}, $x_{m-1}$ and $x_{m}$ do not vanish
simultaneously. Therefore, $x_{m}=1$ may be assumed. Furthermore, if $X_{g(l)}=g(X_l)$ then $g(P)=P$ yields $x_{g(l)}=cx_l$ for some $c\in \mathbb{K}$ and $l=1,\ldots,m$. In particular, since $g(X_{m})=X_{m}$, we
have $x_{m}=cx_{m}$ whence $c=1$. As $g$ is nontrivial, the set $\mathcal{M}=\{l\mid l\neq g(l), 1\le l \le m-2\}$ is non empty.
By $c=1$,  $l\in \mathcal{M}$ yields that $x_{g(l)}=x_{l}$. By (i) and (ii) of Lemma  \ref{lem11octB}. this implies $|\mathcal{M}|\le 2$. Since also $|\mathcal{M}|\ge 2$ holds, the claim follows.
To show the other claims, observe that the transposition $g=(X_l X_n)$ is the involutory homology of ${\rm{PG}}(m-1,\mathbb{K})$ associated with the $(0,1)$ matrix $\{g_{u,v}\}$ whose $1$ entries are $g_{u,u}$ for $u\in
\{1,2,\ldots, m\}\setminus\{l,n\}$, $g_{l,n}$ and $g_{n,l}$. In particular, its axis is the hyperplane $\Pi$ of equation $X_{l}=X_{n}$, and its center is the point $C=(0:0:\cdots:-1:\cdots :0: \cdots :1:\cdots :0)$.
Moreover, $x_{l}=x_{n}$ as $P$ is fixed by $g$. From Lemma \ref{lem11octA}, $x_{l}=x_{n}=1$ may be assumed. This together with (i) of Lemma \ref{lem11octB} yield that $x_u\neq x_v$ for $1\le u<v \le m$ and
$(u,v)\neq(i,j).$ Therefore, the images of $P$ under the action of the stabilizer $H$ of $X_{l}$ and $X_{n}$ in $G$ are all pairwise distinct and their number is $|H|=(m-2)!$.
Since $H$ preserves both $V$ and $\Pi$, all the images of $P$ are in $V\cap \Pi$. Thus $|V\cap \Pi|\ge (m-2)!$. On the other hand, Lemma \ref{lem8oct} shows $|V\cap \Pi|\le (m-2)!$ whence the second part of the claim
follows.
\end{proof}
\begin{lemma}
\label{lem5oct} $V$ has a unique component of dimension $1$.
\end{lemma}
 \begin{proof} From Lemma \ref{lem1D}, the irreducible components of $V$ have dimension at most $1$, and at least one of them is an absolutely irreducible curve $\cC$. Let $P$ be a common point of $\cC$ and $\Pi_\infty$. By the
 second claim of Lemma \ref{lem8octU}, we can assume $P=P_{\varepsilon}$ for a $(m-1)$-th primitive root of unity $\varepsilon$.
 Since $P$ is a nonsingular point of $V$, $\cC$ is the unique irreducible component of $V$ containing $P$. The last claim in Lemma \ref{lem8octU} ensures that this holds for each point in $V\cap \Pi_\infty$. In
 particular, the $1$-dimensional irreducible components of $V$ are exactly the images of $\cC$ under the action of $G$. Clearly, these absolutely irreducible curves $\cC=\cC_1,\cC_2,\ldots,\cC_l$ have the same degree
 $k=\deg(\cC)$, and $kl=\deg(V)$.
 The hyperplane $\Pi$ of equation $X_{m}=X_{m-1}$ meets $\cC$ nontrivially, and let $R\in \cC\cap \Pi$. Clearly the transposition $g=(X_{m}X_{m-1})$ in $G$  which fixes each $X_i$ for $1\le i \le m-2$ and interchanges
 $X_{m}$ with $X_{m-1}$ fixes $R$. Furthermore, $\cC$ contains a point $S\in V$ lying on the hyperplane of equation $\Pi_\omega :X_{m-1}=\omega X_{m}$. By Lemma \ref{lem10octAchar0}, $S\in
 \mathcal{O}_\omega$ and some element $u\in G_S$  acts on the basis $(X_1,X_2,\ldots,X_{m})$ as a $m$-cycle. Since $S$ is a nonsingular point, $u$ preserves $\cC$. Indeed, assume on the contrary that $u$ takes $\cC$ to $\mathcal{C}_i$, $i\neq 1$. Then, since $u$ fixes $S$, $S$ must be in the intersection $\cC\cap \mathcal{C}_i$, a contradiction with the non-singularity of $S$.
 Therefore $g$ and $u$ are automorphisms of $\cC$. By a well known result, the group generated by the transposition $g$ and the $m$-cycle $u$ is the whole symmetric group $G={\rm{Sym}}_{m}$.
 This shows that $G$ preserves $\cC$ and hence $l=1$.
 \end{proof}

From now on, the irreducible non-singular curve $\cC$ stands for the unique $1$-dimensional component of $V$. As $G$ preserves $\cC$ and $G\le {\rm{PGL}}(m,\mathbb{K})$, $G$ is a $\mathbb{K}$-automorphism group of $\cC$.
Since every hyperplane meets $\cC$, the final claim of Lemma \ref{lem10octAchar0} shows that some point of $\mathcal{O}_\omega$ (and hence $P_\omega$) is in $\cC$. Therefore, the $G$-orbit $\Omega_\omega$ of $P_\omega$
 has   length $|G|/m=(m-1)!$.  Similarly, Lemmas \ref{lem7oct} and \ref{lem8octU} yield  that $P_\varepsilon \in \cC$, and hence the $G$-orbit $\Omega_\varepsilon$ of $P_\varepsilon$ has length
 $|G|/(m-1)=m(m-2)!$. A third short $G$-orbit $\Omega_\theta$ arises from transpositions, as Lemma \ref{lem18oct} shows that every transposition of $G$ fixes a point of $\cC$.

 \begin{lemma}
 \label{lem14octC}  The hyperplane of equation $X_1+X_2+\ldots+X_{m}=0$ is the unique hyperplane of ${\rm{PG}}(m-1,\mathbb{K})$ which contains $\cC$.
\end{lemma}
\begin{proof} We have $P_\omega\in \cC$. Since $G$ preserves $\cC$, this yields that $Q=(\omega:1:\omega^2:\cdots:\omega^{m-1})$ is in $\cC$, as well. Assume that $\cC$ is contained in a hyperplane $\Pi$ of equation
$\alpha_1X_1+\alpha_2X_2+\ldots+\alpha_{m}X_{m}=0$ with $\alpha_i\in \mathbb{K}$. Then $\alpha_1+\omega\alpha_2+\omega^2\alpha_3+\ldots+\omega^{m-1}\alpha_{m-1}=0$ and
$\omega\alpha_1+\alpha_2+\omega^2\alpha_3+\ldots+\omega^{m-1}\alpha_{m-1}=0$ whence $\alpha_1=\alpha_2$ by subtraction. Similar argument shows that any two consecutive coefficient in the equation of $\Pi$ are equal. This
yields $\alpha_1=\alpha_2=\ldots=\alpha_{m-1}$, that is $\Pi$ has equation $X_1+X_2+\ldots+X_{m}=0$. On the other hand, as we have already pointed out in Section \ref{back}, the first equation in System (\ref{sy}) shows
that hyperplane of equation $X_1+X_2+\ldots+X_{m}=0$ contains $\cC$.
\end{proof}
 \begin{lemma}
\label{lem11octC} Every transposition of $G$ fixes exactly $(m-2)!$ points of $\cC$.
\end{lemma}
 \begin{proof} We have already pointed out in the proof of Lemma \ref{lem5oct} that the transposition $g=(X_{m-1}X_{m})$  fixes a point $R\in \cC\cap \Pi$ where $\Pi$ is the hyperplane of equation $X_{m}=X_{m-1}$. By
 Lemma \ref{lem11octA} and \ref{lem11octB}, we can assume $R=(x_1:\ldots:x_{m-2}:1:1)$ and that $x_1,\dots,x_{m-2}$ are pairwise distinct and non zero. Therefore any point whose first $m-2$ coordinates are a permutation
 on $\{x_1,\ldots,x_{m-2}\}$ is also fixed by $g$. Thus $g$ has at least $(m-2)!$ fixed points. On the other hand, if $P\in\cC$ is not on $\Pi$ then the last two coordinates of $P$ are different and hence $P$ is not fixed
 by $g$.
 Therefore, the claim holds for $g$. Since the transpositions are pairwise conjugate in $G$, the claim holds true for every transposition in $G$.
 \end{proof}
 \begin{lemma}
 \label{lem18octA} Let $P$ be a point of $\cC$ which is fixed by a transposition $g\in G$. Then the tangent $\ell$ to $\cC$ at $P$ is the line joining $P$ with the center of $g$.
 \end{lemma}
 \begin{proof}
 W.l.o.g. we can assume $g=(X_{m-1}X_{m})$. Then $g$ is an homology with center $C=(0:0:\cdots:0:-1:1)$ and axis the (pointwise fixed) hyperplane $\Pi$ of equation $X_{m}=X_{m-1}$. Also,
 $P=(x_1:\cdots:x_{m-2}:1:1)$.
 The tangent line $\ell$ is the intersection of the  tangent hyperplanes in $P$ of the hypersurfaces of equation $f_i=0$ of system \eqref{sy}, $i=1,\dots,m-2$, and hence its equation is given by
 \begin{equation}\label{sys11}
\left\{
\begin{array}{llll}
X_1 + X_2 + \ldots + X_{m-2}+X_{m-1}+X_{m}=0;\\
x_1X_1  + x_2X_2+\ldots +x_{m-2}X_{m-2}+X_{m-1}+X_{m}=0;\\
\cdots\cdots\\
\cdots\cdots\\
x_1^{m-2-1}X_1 + x_2^{m-2-1}X_2+ \ldots+x_{m-2}^{m-2-1}X_{m-2}+X_{m-1}+X_{m}=0.
\end{array}
\right.
\end{equation}
Since the coordinates of $C$ satisfy system \eqref{sys11}, the claim follows.
 \end{proof}
 \begin{lemma}
 \label{lem19oct} Let $H$ be the stabilizer of $X_i$ and $X_j$ in $G$. Then the quotient curve $\tilde{\cC}$ of $\cC$ with respect to $H$ is rational.
 \end{lemma}
 \begin{proof} Up to a reordering of the coordinates, we can assume $(i,j)=(m-1,m)$. The hyperplanes $\Pi_{\lambda,\mu}$ of equations $\lambda X_{m-1}+\mu X_{m}=0$ form the pencil through the intersection $\Sigma$ of
 the hyperplanes $X_{m-1}=0$ and $X_{m}=0$.
$\Sigma$ is a subspace of ${\rm{PG}}(m,\mathbb{F})$ of codimension $2$, and it is disjoint from $\cC$  by Lemma \ref{lem11octA}. Furthermore, $H$ preserves each $\Pi_{\lambda,\mu}$. Take any point of $P\in\cC$ whose stabilizer $H_P$ in $H$ is trivial. Then the $H$-orbit $\Delta$ of $P$ has
length $(m-2)!$ and $\Delta$ is contained in the unique hyperplane  $\Pi_{\lambda,\mu}$ of the pencil which contains $P$. Since $\deg(\cC)=(m-2)!$, $\Delta$ coincides with the intersection of $\cC$ with
$\Pi_{\lambda,\mu}$.
If the stabilizer $H_P$ of $P\in \cC$ in $H$ is nontrivial, from Lemma \ref{lem18oct} and the proof of Lemma \ref{lem11octC}, then $|H_P|=2$ and the only nontrivial element in $H_P$ is a transposition $h$. Indeed if $h'$
is a nontrivial element of $H_P$, from Lemma \ref{lem18oct} it is a transposition $(X_i X_j)$, and from the proof of Lemma \ref{lem11octC}, $x_i=x_j$ and $x_{m-1}=x_{m}$, which is a contradiction with Lemma
\ref{lem11octB}.
From Lemma \ref{lem18octA}, the tangent line $\ell$ to $\cC$ at $P$ contains the center $C$ of $h$. The hyperplane $\Pi$ of equation $X_{m-1}-X_{m}=0$ is the axis of the transposition $g$ which interchanges $X_{m}$ and $X_{m-1}$. Since $g$ and $h$ commute, it follows that $C\in \Pi$. Therefore, $\ell$ is contained in $\Pi$, and hence $I(P,\cC\cap \Pi)\ge 2$, where $I(P,\cC\cap \Pi)$ is the intersection multiplicity of $\cC$ and $\Pi$ at $P$. From the higher dimensional generalization of
B\'ezout's theorem,
$|\cC\cap \Pi|\le \ha (m-2)!$. Thus, $|\cC\cap \Pi|= \ha (m-2)!$, and, again,  $\Delta$ coincides with $\cC\cap \Pi$. This shows that $\tilde{\cC}$ is isomorphic to the rational curve which is the projection of $\cC$ from
the vertex $\Sigma$.
 \end{proof}
\begin{lemma}
\label{lem11octE} Let $\gg$ be the genus of  $\cC$. Then
\begin{equation}
\label{eq11octB} 2\gg-2= \ha ((m-2)(m-3)-4)(m-2)!.
\end{equation}
\end{lemma}
\begin{proof}
Let $H\cong \Sym_{m-2}$ be the subgroup of $G$ which fixes both $X_{m}$ and $X_{m-1}$. From Lemma \ref{lem18oct}, if $h\in H$ has a fixed point in $\cC$ then $h$ is a transposition.
If $P$ is a fixed point of a transposition then, up to a reordering of the coordinates,  
$P=(1,1,x_3,\dots,x_{m-1},x_{m})$. A point whose coordinates are a permutation of those of $P$ is another fixed point of the transposition. Among these points, those which are contained in the hyperplane
$X_{m-1}=X_{m}$, fixed by $H$, are as many as $(m-2)!$, and no more, since $\deg(\mathcal{C})=(m-2)!$.
Therefore, the number of short orbits of $H$ is equal to the number of choices of two values among $x_3,\dots,x_{m}$. As these are $m-2$ such distinct values, the short orbits are as many as $(m-2)(m-3)$. The claim follows from the Riemann-Hurwitz formula.
\end{proof}
\begin{lemma}
\label{lem12oct} $V$ is irreducible, that is, $V=\cC$.
\end{lemma}
\begin{proof} Suppose on the contrary the existence of a point $Q=(q_1:\cdots:q_{m-1}:q_{m})$ of $V$ which is not in $\cC$. Since $(1:0:0:\cdots:0)$ is not a point of $V$, both $q_{m}=1$ and $q_{m-1}=\lambda\neq 0$
may be assumed. Then $Q$ is contained in the hyperplane $\Pi$ of equation $X_{m-1}=\lambda X_{m}$. Choose a point $P\in \cC$ lying on $\Pi$. The stabilizer $H$ of $X_{m-1}$ and $X_{m}$ in $G$ preserves $\Pi$. Since
$H$ also preserves $\cC$, the $H$-orbit $\Delta$ of $P$ is contained in $\Pi$. As $P\in\Pi$, this together with Lemma \ref{lem8oct} yield  $|\Delta|<|H|$, that is, $H_P$ is nontrivial. As in the proof of Lemma
\ref{lem19oct}, from Lemma \ref{lem18oct} and Lemma \ref{lem11octC} we obtain $|H_P|=2$. Therefore, $|\Delta|=\ha (m-2)!$. Thus, the higher dimensional generalization of B\'ezout's theorem yields that the length of the
$H$-orbit of $Q$ is also less than $
(m-2)!$, that is, $H_Q$ is non-trivial. This is a contradiction, since Lemma \ref{lem18oct} together with Lemma \ref{lem11octC}, yields that all the fixed points of any transposition of $G$ are on $\cC$.\end{proof}
Lemmas \ref{lem1} and \ref{lem12oct} have the following corollary.
\begin{theorem}
\label{the2502} $V$ is an irreducible non-singular curve of ${\rm{PG}}(m-2,\mathbb{K})$.
\end{theorem}
\section{The action of $G$ on $V$}
\label{secautgroup}
As we have already pointed out, $G$ has at least three short orbits on $V$, named $\Omega_\omega$, $\Omega_\varepsilon$ and $\Omega_\theta$. We prove that
they are the only short $G$-orbits.
\begin{lemma}\label{3orbite}
The short $G$-orbits on $V$ are exactly $\Omega_\omega$, $\Omega_\varepsilon$ and $\Omega_\theta$, and they have lengths $(m-1)!$, $m(m-2)!$ and $m!/2$.
\end{lemma}
\begin{proof}
Since $|\Omega_\omega|=(m-1)!$, $|\Omega_\varepsilon|=m(m-2)!$ and $\Omega_\theta$ are short $G$-orbits, and $|\Omega_\theta|\leq \ha m!$, the Hurwitz genus formula (\ref{hufo}) applied to the $G$ yields
\begin{equation*}
2\gg-2\geq -2m!+(m!-(m-1)!)+(m!-m(m-2)!)+(m!-\textstyle\frac{1}{2}m!).
\end{equation*}
Comparison with \eqref{eq11octB} shows that equality holds. Therefore, no further short $G$-orbits on $V$ exists, and $|\Omega_\theta|=\ha m!$, that is, the stabilizer  of a point on $V$ which is fixed
by a transposition contains no more nontrivial element of $G$.
\end{proof}

\begin{lemma}
\label{lem21jun21} The short $G$-orbit $\Omega_\omega$ consists of the common points of $V$ and the hypersurface $\Sigma_{m-1}$
of equation
\begin{equation}
\label{eq21jun21} X_1^{m-1}  +  X_2^{m-1}  +  \ldots+X_{m}^{m-1}=0.
\end{equation}
\end{lemma}
\begin{proof} We observe first that no point of the $G$-orbit $\Omega_\varepsilon$ is in $\Sigma_{m-1}$. In fact, if $P=(\xi_1:\ldots:\xi_{m-1}:0)\in \Omega_\varepsilon$ then $\xi_j^{m-1}=1$ and hence $\sum_{j=1}^{m-1}\xi_j^{m-1}=m-1\neq 0$, thus $P\not\in \Sigma_{m-1}$. On the other hand, it is readily seen that $P\in \Sigma_{m-1}$ for any $P\in \Omega_{\omega}$. From the higher dimensional generalization of B\'ezout's theorem, $|V\cap \Sigma_{m-1}|\le (m-1)!$. Furthermore, since $G$ also preserves $\Sigma_{m-1}$, the intersection $V\cap \Sigma_{m-1}$ is $G$-invariant, as well. Therefore, Lemma \ref{3orbite} yields $V\cap \Sigma_{m-1}=\Omega_\omega$.
\end{proof}
A similar argument can be used to prove the following lemmas.
\begin{lemma}
\label{lem21jun21bis} The short $G$-orbit $\Omega_\varepsilon$ consists of the common points of $V$ and the hypersurface $\Sigma_{m}$
of equation
\begin{equation}
\label{eq21jun21bis} X_1^{m}  +  X_2^{m}  +  \ldots+X_{m}^{m}=0.
\end{equation}
\end{lemma}
\begin{lemma}
\label{lem21jun21ter} The short $G$-orbit $\Omega_\theta$ consists of the common points of $V$ and the hypersurface $\Sigma_{m(m-1)/2}$
of equation
\begin{equation}
\label{eq21jun21ter} X_1^{m(m-1)/2}  +  X_2^{m(m-1)/2}  +  \ldots+X_{m}^{m(m-1)/2}=0.
\end{equation}
\end{lemma}

\begin{lemma}
\label{le8mar} Let $P_\omega$ be the point of $V$ given in Lemma \ref{lem10octAchar0}. Then the number of fixed points on $V$ of the involution in the stabilizer of $P_\omega$ in $G$ is
\begin{equation}\label{eq8mar}
2^{m/2}\frac{(m/2)!}{m},
\end{equation}
if $m$ is even, and
\begin{equation}\label{eq8mar1}
2^{(m-1)/2}\frac{((m-1)/2)!}{m},
\end{equation}
if $m$ is odd.

\end{lemma}
\begin{proof} We will perform the proof for $m$ even. The same approach can be applied for $m$ being odd. Let $u$ be the (unique) involution in $G$ which fixes $P_\omega$. Then $u$ acts on $(X_1,\ldots,X_{m})$ as the involutory permutation $(X_1X_{(m+2)/2})(X_2X_{(m+4)/2})\cdots (X_{m/2}X_{m})$.
Therefore,
the centralizer $C_G(u)$ of $u$ in $G$ has order  $2^{m/2}(m/2)!$, and hence $u$ has as many as
$$k=\frac{m!}{2^{m/2}(m/2)!}$$
conjugate in $G$. We show that if $v$ is conjugate of $u$ in $G$ and $u\neq v$ then $u$ and $v$ has no common fixed point. Assume on the contrary the existence of a point $Q\in V$ such that $u(Q)=v(Q)$. Then $u$ and $v$
are two distinct involutions in the stabilizer $G_Q$ of $Q$ in $G$. On the other hand, since either $p=0$ or $p>0$ and $p\nmid |G|$, $G_Q$ is cyclic, and hence it contains at most one involution; a contradiction. Therefore, each point in the
$G$-orbit $\Omega_\omega$ is the fixed point of exactly one involution which is conjugate to $u$ in $G$. If $N_u$ counts the fixed points of $u$ in $\Omega_\omega$,
this yields that $|\Omega_\omega|=k N_u$. Therefore, the number of fixed points $u$ in $\Omega_\omega$ equals (\ref{eq8mar}). It remains to show that $u$ has no further fixed points on $V$. By Lemma \ref{lem7oct} the
stabilizer of any point $Q\in \Omega_\varepsilon$ in $G$ has odd order and hence contains no involution. By Lemma \ref{3orbite}, the 1-point stabilizer of the remaining short orbit has order $2$ and its non-trivial
element is a transposition. Since $u$ is not a transposition, the claim follows.
\end{proof}
\begin{theorem}
\label{th26062024B} $G$ is an irreducible group in $PG(m-2,\mathbb{K})$.
\end{theorem}
\begin{proof} Assume on the contrary that $G$ preserves a projective subspace $\Pi$ of $PG(m-2,\mathbb{K})$. Any homology $g\in G$ induces a projectivity $\bar{g}$ of $\Pi$ where $\bar{g}$ is either the identity or a homology of $\Pi$ and in the latter case $g$ and $\bar{g}$ have the same center. From Lemma \ref{lem18oct}, the transpositions in $G$ are involutory homologies in $G$, and hence such involutory homologies form a conjugacy class $\cD$ in $G$. Therefore, two cases occur, namely either each involutory homology in $\cD$ fixes $\Pi$ pointwise, or the center of each involutory homology lies in $\Pi$. From the proof of Lemma \ref{lem18oct}, in the former case $\Pi$ is contained in any hyperplane $\Pi_{ij}$ of equation $X_i=X_j$ with $1\le i < j \le m$ but these hyperplanes have no common point, a contradiction; in the latter case, $\Pi$ contains all points  $C_{ij}=(0:0:\cdots:-1:\cdots :0: \cdots :1:\cdots :0)$ with $1\le i < j \le m$, but these points generate $PG(m-2,\mathbb{K})$, a contradiction.
\end{proof}

These results on the action of $G$ also allow to show that every solution of {\rm{(\ref{sy})}} satisfies further diagonal equations.

\begin{proposition}
\label{prop10.03} Every solution of {\rm{(\ref{sy})}} also satisfies the following equations.
\begin{equation}
\label{syplus}
\left\{
\begin{array}{lll}
X_1^{m+2}  +  X_2^{m+2}  +  \ldots+X_{m}^{m+2}=0,\\
X_1^{m+3}  +  X_2^{m+3}  +  \ldots+X_{m}^{m+3}=0,\\
\cdots\cdots\\
\cdots\cdots\\
X_1^{2m-3}  +  X_2^{2m-3}  +  \ldots+X_{m}^{2m-3}=0.
\end{array}
\right.
\end{equation}
\end{proposition}
\begin{proof}
Let $P=(a_1,\ldots,a_{m})$ be any point of $V$. If $P\in \Omega_\omega$ then $a_i=a_i^{m+1}$ and hence $a_i^{m+1+l}=a_i^{1+l}$ for any positive integer $l$. Therefore, the claim holds for every $P\in \Omega_\omega$.

Let $\cH_j$ be the (irreducible) hypersurface of equation $X_1^{m+1+j}  +  X_2^{m+1+j}  +  \ldots+X_{m}^{m+1+j}=0$. Then $\Omega_\omega$ is contained in $\cH_j$. We prove that $\cH_j$ contains $V$ as far as $j\le m-4$. Assume
on the contrary that this does not occur for some $j$. Then Lemma \ref{lem8oct} together with B\'ezout's theorem applied to the intersection of $\cH_j$ with $V$ yield
\begin{equation}
\label{eqbez}
\sum_{Q\in \cH\cap V}I(Q,\cH\cap V)=(m+1+j)(m-2)!.
\end{equation}
We show that $\cH_j$ contains some points of $V$ other than those in $\Omega_\omega$. As $G$ also preserves $\cH_j$, the intersection number $I(Q,\cH_j\cap V)$ is invariant when $Q$ ranges over an $G$-orbit. For a point
$Q\in \Omega_\omega$, let $\lambda=I(P,\cH_j\cap V)$. Then
$\sum_{Q\in \Omega_\omega}I(Q,\cH_j\cap V)=\lambda (m-1)!$ whence $\lambda(m-1)\le m+1+j$. For $\lambda\ge 2$, this would yield $2(m-1)\le m+1+j$, that is, $j>m-4$, a contradiction. Therefore $\lambda=1$. Since
$(m+1+j)(m-2)!>(m-1)!$, the claim follows.

Let $R$ be a point $R\in \cH\cap V$ not in $\Omega_\omega$. From Lemma \ref{3orbite}, the $G$-orbit of $R$ has length at least $m(m-2)!$. Then $\cH_j$ and $V$ have at least $((m-1)+m)(m-2)!=(2m-1)(m-2)!$ common points.
Since $j\le m-3$ implies $2m-1>m+1+j$, this contradicts (\ref{eqbez}).
\end{proof}
\begin{lemma}\label{lambda12}
The group $G_{m,m-1,m-2}$ acts on $\Omega_\theta$ with $\lambda_1$ short orbits and $\lambda_2$ long orbits where
\begin{equation*}
\begin{split}
&\lambda_1=(m-2)(m-3)(m-4),\quad \lambda_2=3(m-2)^2.
\end{split}
\end{equation*}
\end{lemma}
\begin{proof}
Fix a point $P\in\Omega_\theta$. As in the proof of Lemma \ref{lem11octC}, assume that $P=(x_1,\dots,x_{m})$ where $x_i=x_j=t$ for some $1\leq i<j\leq m$, $t\in\mathbb{K}$ and $x_k\neq x_l$ whenever
$(k,l)\neq(i,j)$. So, the coordinates $(x_1,\dots,x_{m})$ of $P$ are $m-1$ different values from $\mathbb{K}$. The points of $\Omega_\theta$ are those whose coordinates are a permutation of $x_1,\dots,x_{m}$, that is $Q\in\Omega_\theta$ if and only if $Q=(x_{\sigma(1)},\dots,x_{\sigma(m)})$ for a permutation
$\sigma \in G$.
Two points are in the same short $G_{m,m-1,m-2}$-orbit if and only if they share the last three coordinates, hence a short $G_{m,m-1,m-2}$-orbit arises every time we fix an ordered triple $(x_a,x_b,x_c)$, with $x_a,x_b,x_c\neq t$. This can be
done in $(m-2)(m-3)(m-4)$ different ways.
A long $G_{m,m-1,m-2}$-orbit arises every time we fix an ordered triple $(x_a,x_b,x_c)$, with either one or two of the values $x_a,x_b,x_c$ being equal to $t$. This can be done in $3(m-2)(m-3)$ and $3(m-2)$ different ways respectively.
Therefore $\lambda_1=(m-2)(m-3)(m-4)$ and $\lambda_2=3(m-2)^2$.
\end{proof}

\section{Quotient curves of $V$}
\label{secqc}
We have already determined the quotient curve of $V$ with respect to the subgroup of $G$ which fixes two given coordinates $X_i$ and $X_j$; see Lemma \ref{lem19oct}. In this section, we consider the more general case
where the subgroup $H_l$ of $G$ fixes $l\ge 3$ coordinates. Let $d=m-1-l$. W.l.o.g. these coordinates are assumed to be $X_{d+2},\ldots,X_{m}$.
The hyperplanes $\Pi_i: X_i=0$ with $i=d+2,\ldots,m$ meet in a $d$-dimensional subspace $\Sigma$ which is disjoint from $V$ by Proposition \ref{lem11octA}. Clearly, $H_l$ preserves $\Sigma$. Furthermore, the hyperplanes $\Pi_i: X_i=0$ with $i=1,\ldots,d+1$ meet in a $(m-d-2)$-dimensional subspace $\Sigma'$ disjoint from $\Sigma$. Clearly, $H_l$ fixes $\Sigma'$ pointwise.
Projecting $V$ from $\Sigma$ on $\Sigma'$ produces a curve $\bar{V}$ of $\Sigma'$ whose degree is equal to $(m-2)!/(m-l)!$.
\begin{proposition}
\label{pro5mar} Let $H_l$ be the stabilizer of $(X_{j_1},\ldots,X_{j_l})$ in $G$. If $l\ge 2$ then the quotient curve $V/H_l$ is isomorphic to $\bar{V}$.
\end{proposition}
\begin{proof}
Take a point $P\in V$ such that no nontrivial element of $H_l$ fixes $P$. Then the $H_l$-orbit $\Omega$ of $P$ has length $(m-l)!=(d+1)!$. On the other hand,  Lemma \ref{lem10giugno} shows that $\Sigma$ and $P$ generate a $(d+1)$-dimensional subspace $\tilde{\Sigma}$ that cuts out on $V$ a set of $(d+1)!$ points. Since $\Omega$ is contained in $\tilde{\Sigma}$, it turns out that $\Omega=V\cap \Sigma'$ whence the claim follows.
\end{proof}
\begin{rem}
\label{rem11giugno} {\emph{Proposition \ref{pro5mar} shows that $\Sigma$ is an outer Galois subspace of $V$. }}
\end{rem}
\begin{proposition}
\label{pro11giugno} Let $\bar{\gg}$ be the genus of the quotient curve $V/H_l$ where $H_l$ is the stabilizer of $(X_{j_1},\ldots,X_{j_l})$. If $l\ge 2$ then
\begin{equation}
\label{eq11giugno} 2\bar{\gg}-2=  \frac{(m-2)(m-3)-4-(m-l)(m-1-l)}{2(m-l)!}(m-2)!
\end{equation}
\end{proposition}
\begin{proof} The number of transpositions in $H_l$ is equal to $\ha(m-l)(m-1-l)$.  By Lemma \ref{lem11octC} each such transposition  has as many as $(m-2)!$ fixed points. From Lemmas \ref{lem10octAchar0}, \ref{lem7oct}, and \ref{3orbite}, no nontrivial element of $H_l$ other than its transpositions fixes a point of $V$. Therefore, the claim follows from the  Hurwitz genus formula (\ref{hufo}).
\end{proof}

\subsection{Quotient curve of $V$ by the $3$-coordinate stabilizer of $G$}
\label{sq} In this section $H=G_{m,m-1,m-2}$ is the stabilizer of $X_{m}, X_{m-1}, X_{m-2}$ in $G$, and  $\cX$ is the quotient curve of $V$ with respect to $H$. Then $\cX$ is an irreducible plane curve of degree $m-2$ whose genus equals $\ha(m^2-7m+12)$ by (\ref{eq11giugno}). Hence $\cX$ is non-singular. The following proposition shows that $\cX$ coincides with the curve introduced and investigated in \cite{voloch}.
\begin{theorem}
\label{the24jun} $\cX$ has homogeneous equation $G_{m-2}(x,y,z)=0$ where
\begin{equation}\label{eqvoloch}
G_{m-2}(x,y,z)=\sum_{i,j,k\geq0,i+j+k=m-2}x^iy^jz^k.
\end{equation}
\end{theorem}
\begin{proof} Since $|H|=(m-3)!$, Lemma \ref{lem10octAchar0} yields that $\cX$ contains as many as $(m-1)(m-2)$ points $(\alpha:\beta: 1)$ with
$\alpha^{m}=\beta^{m}=1$ but $\alpha,\beta \ne 1$. Let $\tilde{\cX}$ be the plane curve with homogeneous equation $G_{m-2}(x,y,z)=0$. By  \cite[Equation (2)]{voloch}
\begin{equation}\label{Gmeq2}
G_{m-2}(x,y,z)=\frac{1}{x-y}\left(\frac{x^{m}-z^{m}}{x-z}-\frac{y^{m}-z^{m}}{y-z}\right).
\end{equation}
This shows that $\tilde{\cX}$ also contains each point $(\alpha:\beta: 1)$ with $\alpha^{m}=\beta^{m}=1$ but $\alpha,\beta \ne 1$. Therefore,
$\cX$ and $\tilde{\cX}$ have at least $(m-1)(m-2)$ pairwise distinct points. On the other hand, since $\mathcal{X}$ and $\tilde{\cX}$ both have degree $m-2$, B\'ezout's theorem applied to $\cX$ and $\tilde{\cX}$ yields that if $\cX$ and $\tilde{\cX}$  were distinct then they could  share at most $(m-2)^2$ points. Therefore, $\cX=\tilde{\cX}$.
\end{proof}
Theorem \ref{the24jun} shows that $\mathbb{K}(\cX)=\mathbb{K}(x,y)$ with $G_{m-2}(x,y,1)=0$. From Lemma \ref{lem19oct}, the quotient curve of $V$ with respect to the stabilizer of $X_{m},X_{m-1}$ is rational. Therefore its function field is $\mathbb{K}(x)$.

\begin{theorem}
\label{th200721} 
The Galois closure $M$ of $\mathbb{K}(\cX)|\mathbb{K}(x)$ is $\mathbb{K}(V)$, with Galois group isomorphic to ${\rm{Sym}}_{m-2}$.
\end{theorem}
\begin{proof} Since $\mathbb{K}(V)|\mathbb{K}(x)$ is a Galois extension, $M$ may be assumed to be a subfield of $\mathbb{K}(V)$. Thus
$\mathbb{K}(V)|M$ is a Galois extension, as well. Galois theory yields that ${\rm{Gal}}(\mathbb{K}(V)|M)$ is a normal subgroup of ${\rm{Gal}}(\mathbb{K}(V)|\mathbb{K}(x))$. Since ${\rm{Gal}}(\mathbb{K}(V)|\mathbb{K}(x))\simeq {\rm{Sym}}_{m-2}$, for $ m\ge 7$ and $m=5$, the unique non-trivial normal subgroup of ${\rm{Gal}}(\mathbb{K}(V)|\mathbb{K}(x))$ is ${\rm{Alt}}_{m-2}$.
On the other hand, as $\mathbb{K}(\cX)$ is the fixed field of $G_{m,m-1,m-2}\cong {\rm{Sym}}_{m-3}$, $\mathbb{K}(\cX)\subseteq M\ne \mathbb{K}(V)$ would imply that  ${\rm{Alt}}_{m-2}$ is isomorphic to a subgroup of ${\rm{Sym}}_{m-3}$ which is impossible by $|{\rm{Sym}}_{m-3}|<|{\rm{Alt}}_{m-2}|$. If $m=6$, a further case arises, namely ${\rm{Gal}}(\mathbb{K}(V)|M)$ is normal in ${\rm{Alt}}_4$ and of order $4$. However, this is again impossible since $4\nmid |{\rm{Sym}}_3|=6$.
Thus $M=\mathbb{K}(V)$ and  ${\rm{Gal}}(\mathbb{K}(\cX)|\mathbb{K}(x))$ is
isomorphic to the subgroup of $G$ fixing $x$. Since this subgroup is isomorphic to the stabilizer of $X_{m},X_{m-1}$ in $G$, we have ${\rm{Gal}}(\mathbb{K}(\cX)|\mathbb{K}(x))\cong {\rm{Sym}}_{m-2}$.
\end{proof}

\section{The automorphism group of $V$}

We are now in position to prove that $G$ is the whole automorphism group of $V$ whenever $\aut(V)$ is tame, and in particular in zero characteristic. However, as Theorem \ref{th26062024A} shows, this is not true in positive characteristic, in which case $\aut(V)$ may contain additional elements of order $p$.
\begin{theorem}
\label{the171021} If $\mathbb{K}$ has zero characteristic, or it has positive characteristic and the $\mathbb{K}$-automorphism group of $V$ is tame, then $G$ is the $\mathbb{K}$-automorphism group of $V$.
\end{theorem}
\begin{proof}
For $m=5$, $V$ is the characteristic $p$ version of the Bring curve of genus $4$ whose automorphism group is isomorphic to ${\rm{Sym}}_5$. Therefore, we may limit ourselves to the cases where $m\ge 6$.

By way of a contradiction assume that $G={\rm{Sym}}_{m}$ is a proper subgroup of the $\mathbb{K}$-automorphism group $\Gamma$ of $V$.
Then two cases arise, according as $G$ is a normal subgroup of $\Gamma$ or is not.

In the former case, assume that the centralizer $C_\Gamma(G)$ of $G$ in $\Gamma$ is trivial. Then for any $\gamma\in\Gamma$ the map $g\mapsto \gamma^{-1}g\gamma$ is a non
trivial automorphism of $G$. Hence $\Gamma$ is isomorphic to a subgroup of $\Aut(G)$. However, if $m\ne 6$, then $\Aut(G)\cong G$, see for instance \cite[Section 2.4]{RW}. Hence $G=\Gamma$, a contradiction.
In the remaining case, $m = 6$, then $G\cong P\gamma L(2, 9)$ and $\Aut(G)\cong P\Gamma L(2,9)$; see \cite[Section 2.4.2]{RW}. Therefore, $[\Gamma : G] = 2$. Lemma \ref{3orbite} shows that $G$ has a unique orbit of length
$(m-1)! = 120$, namely $\Omega_\omega$. Since $G$ is a normal subgroup of $\Gamma$, this yields that $\Gamma$ also preserves $\Omega_\omega$. Hence,
the stabilizer of $P_\omega \in \Omega_\omega$ of $\Gamma$ has order $12$. This yields that $\Gamma$ has a cyclic subgroup of order $12$, but this is
impossible since $P\Gamma L(2, 9)$ has two subgroups of order $12$, up to conjugation, but neither is cyclic.

Otherwise, $C_\Gamma(G)$ is a nontrivial subgroup of $\Gamma$ which is disjoint from $G$ as $G$ has trivial center. Furthermore, since $C_\Gamma(G)$ is contained in the normalizer of $G$ in $\Gamma$, $C_\Gamma(G)$ induces a permutation group on the set of the short
orbits of $G$. By Lemma \ref{3orbite}, there are three such orbits and they have pairwise different lengths. Therefore, this permutation group is trivial, that is, $C_\Gamma(G)$ preserves each of the three short orbits of
$G$. Also, $C_\Gamma(G)\cong C_\Gamma(G)G/G$  is a isomorphic to a nontrivial $\mathbb{K}$-automorphism group $\bar{G}$ of the quotient curve $\bar{V}=V/G$ which fixes three points of $\bar{V}$. However, this is impossible as $\bar{V}$ is rational by Lemma \ref{lem19oct}.

In the case where $G$ is not normal in $\Gamma$, take any conjugate $\tilde{G}=\gamma^{-1}G\gamma$ of $G$ with $\gamma\in\Gamma$, and consider their intersection $N=G\cap \tilde{G}$. Assume first  $N={\rm{Alt}}_{m}$. If this occurs for every
$\gamma \in \Gamma$, then $N$ is a normal subgroup of $\Gamma$. As shown before, this is impossible. Now take $\tilde{G}$ such that $N\neq {\rm{Alt}}_{m}$. Then $|N|\le (m-1)!$; see \cite[Theorem 5.2B]{DM}. Therefore,
$|<G,\tilde{G}>|\ge |G|^2/(m-1)!=m\cdot m!$, whence $|\Gamma|\ge m\cdot m!$. A comparison with (\ref{eq11octB}) gives
$$ \frac{|\Gamma|}{\gg-1}=\frac{4m^2(m-1)}{(m-2)(m-3)-4}$$
whence $|\Gamma|>84(\gg-1)$ for $m=5,6$ and $m>15$ whereas $40(\gg-1)<|\Gamma|\leq 84(\gg-1)$ for the remaining cases of $m$.

Since $\Gamma$ is assumed to be tame, $\Gamma$ has exactly three short orbits, and the classical Hurwitz bound yields $|\Gamma|\le 84(\gg-1)$. So we are left with $7\leq m\leq 15$. From Hurwitz's proof given in \cite{sti} or
\cite[Theorem 11.56]{HKT}, it follows that $|\Gamma|>40(\gg-1)$ is only possible when $V$ has two points with stabilizers of $\Gamma$ of order $2$ and $3$, respectively. But Lemma \ref{3orbite} shows that the three
nontrivial point stabilizers of $\Gamma$ have order $m,m-1$ and $2$. This completes the proof.
\end{proof}

\label{nta}
\begin{theorem}
\label{th26062024A} If $\aut(V)$ contains some non-tame projectivities then $m=p-1$ and the linear subgroup $L\aut(V)$ of $\aut(V)$ is isomorphic to ${\rm{Sym}}_p$.
\end{theorem}
\begin{proof} Let $g\in \aut(V)$ be a projectivity preserving $V$. We show that $g$ preserves the non-degenerate quadric $\cQ$ of $PG(m-2,\mathbb{K})$ of equations
$$
\begin{cases}X_1 +  X_2 + \ldots + X_m=0,\\
X_1^2  +  X_2^2 + \ldots + X_m^2=0.
\end{cases}
$$
Since $g$ is a projectivity,  Lemma \ref{lem14octC} yields that $g$ preserves the hyperplane of equation $X_1+X_2+\ldots+X_m=0$. Furthermore, $g(\cQ)$ is a quadric of equation $F=F(X_1,X_2,\ldots,X_m)=0$ where $\deg(F)=2$. Since $g\in\aut(V)$,
$$F=u_1(X_1+X_2+\ldots+X_m)+u_1(X_1^2+X_2^2+\ldots+X_m^2)+\ldots+u_{m-2}(X_1^{m-2}+X_2^{m-2}+\ldots+X_m^{m-2})$$
 where $u_1,u_2,\ldots,u_{m-2}\in \mathbb{K}[X_1,X_2,\ldots,X_{m-2}]$.
 Therefore, $u_i=0$ for $2<i\le m-2$.  Now, take any point $P=P(\xi_1,\ldots,\xi_m)\in Q$. Then $F(\xi_1,\xi_2,\ldots,\xi_m)=u_1(\xi_1+\xi_2+\ldots+\xi_m)+u_2(\xi_1^2+\xi_2^2+\ldots+\xi_m^2)=0+0=0$. Therefore, $g(P)\in \cQ$.  
In particular, $L\aut(V)$ contains no elation.

Let $N$ be the subgroup generated by all involutory homologies of $L\aut(V)$. From Lemma \ref{lem18oct}, $G=\rm{Sym}_m$ is generated by involutory homologies and hence $L\aut(V)$ contains $G$. Also,  $N$ is a normal subgroup of $L\aut(V)$. The centralizer of $\rm{Sym}_m$ in $L\aut(V)$ is trivial, since the centers of the involutory homologies in $G$ span $PG(m-2,\mathbb{K})$ as we have pointed out in the proof of Theorem \ref{th26062024B}. Therefore, if $N=G$ is supposed, then $L\aut(V)$ is isomorphic to a subgroup of the automorphism group $\aut({\rm{Sym}}_m)$  of $\rm{Sym}_m$. For $m\ne 6$,  $\aut({\rm{Sym}}_m)\cong {\rm{Sym}}_m$, see for instance \cite[Section 2.4]{RW}. From this, $L\aut(V)\cong \rm{Sym}_m$ whence $L\aut(V)=G$ for $m\ne 6$ but this contradicts the hypothesis that $\aut(V)$ contains non-tame projectivities. In the exceptional case, $|\aut({\rm{Sym}}_6)|=2| {\rm{Sym}}_6|=1440=2^5\cdot3^2\cdot 5$; in particular $p$ does not divide $|L\aut(V)|$, and hence $L\aut(V)=G$ by Theorem \ref{the171021}.      
 Therefore, $G$ is a proper subgroup of $N$.
 Now, we apply Result \ref{zalesA} to $N$. By Lemma \ref{lem11octE} and Henn's classification theorem \cite[Theorem 11.127]{HKT}, inequality (\ref{eq270624}) does not hold in our case. Therefore, $m=p-1$ and $N\cong \rm{Sym}_p$.

We have to show that if $m=p-1$ then $L\aut(V)$ is indeed larger than $\rm{Sym}_m$. Let $m=p-1$. We show that the linear map
 $$\lambda:(X_1,\ldots,X_j,\ldots,X_{p-1})\rightarrow (X_1-X_{p-1},\ldots,X_j-X_{p-1},\ldots,X_{p-2}-X_{p-1},-X_{p-1})$$  preserves $V$.
 By Section \ref{powersum}, in the function field $\mathbb{K}(V)$, $\sigma_k(X_1,X_2,\ldots,X_{p-1})=0$ for $k=1,\ldots,p-3$.
 We will use the following equations
$$
\sigma_k(X_1,\dots,X_{p-2})=\sigma_{k}(X_1,\dots,X_{p-1})-X_{p-1}\sigma_{k-1}(X_1,\dots,X_{p-2}).
$$
Since $\sigma_{1}(X_1,\dots,X_{p-2})=-X_{p-1}$, induction on $k$ gives
\begin{equation}\label{ekp-2}
    \sigma_k(X_1,\dots,X_{p-2})=(-1)^kX_{p-1}^k,
\end{equation}
Write $\sigma_k(X_1-X_{p-1}, X_2-X_{p-1}, \dots, X_{p-2}-X_{p-1}, -X_{p-1})$ as a sum of two terms:
\begin{align*}
&\sigma_k(X_1-X_{p-1}, X_2-X_{p-1}, \dots, X_{p-2}-X_{p-1},- X_{p-1})=\\
&\left(\sum_{1\leq {i_1}<\dots<{i_k}\leq p-2}(X_{i_1}-X_{p-1})\cdots (X_{i_k}-X_{p-1})\right)-{X_{p-1}\left(\sum_{1\leq {i_1}<\dots<{i_{k-1}}\leq p-2}(X_{i_1}-X_{p-1})\cdots (X_{i_k}-X_{p-1})\right)}
\end{align*}
which is equal to
$$\sigma_k(X_1-X_{p-1}, X_2-X_{p-1}, \dots, X_{p-2}-X_{p-1})-X_{p-1}\sigma_{k-1}(X_1-X_{p-1}, X_2-X_{p-1}, \dots, X_{p-2}-X_{p-1}).$$

Next, write $\sigma_k(X_1-X_{p-1}, X_2-X_{p-1}, \dots, X_{p-2}-X_{p-1})$ as a polynomial in $X_{p-1}$. Expanding $\sigma_k$ shows that each product $X_{i_1}\cdots X_{i_j}$, with $X_{i_l}\in\{X_1,\dots,X_{p-2}\}$, appears every time when the remaining $k-j$ factors in the product are equal to $X_{p-1}$. This occurs every time when the remaining $k-j$ indeterminates $X_{i_l}$ are chosen among $\{X_1,\dots,X_{p-2}\}\setminus\{X_{i_1},\dots,X_{i_j}\}$, that is, it appears $\binom{p-2-j}{k-j}$ times. Summing up,
$$
\sigma_k(X_1-X_{p-1}, X_2-X_{p-1}, \dots, X_{p-2}-X_{p-1})=\sum_{j=0}^k (-1)^{k-j}\binom{p-2-j}{k-j}X_{p-1}^{k-j}\sigma_j(X_1,\dots,X_{p-2}).
$$
This together with  \eqref{ekp-2} yield
$$
\sigma_k(X_1-X_{p-1}, X_2-X_{p-1}, \dots, X_{p-2}-X_{p-1})=(-1)^{k}X_{p-1}^{k}\sum_{j=0}^k \binom{p-2-j}{k-j}.
$$
To compute the sum on the right hand side, the following identities will be useful
$$
\binom{n}{k}=\binom{n+1}{k}-\binom{n}{k-1},\qquad \binom{p-1}{k} \equiv (-1)^k \hspace{-0.2cm}\pmod{p}.
$$
Now,
\begin{align*}
&\sum_{j=0}^k \binom{p-2-j}{k-j}=
\sum_{j=0}^{k-1}\left[ \binom{p-2-j+1}{k-j}-\binom{p-2-j}{k-j-1}\right]+\binom{p-2-k}{0}=\\
&\binom{p-1}{k}-\binom{p-2-k+1}{0}+\binom{p-2-k}{0}=\binom{p-1}{k}\equiv (-1)^k \hspace{-0.2cm}\pmod{p};
\end{align*}
Summing up,
\begin{align*}
&\sigma_k(X_1-X_{p-1}, X_2-X_{p-1}, \dots, X_{p-2}-X_{p-1},- X_{p-1})=\\
&\sigma_k(X_1-X_{p-1}, X_2-X_{p-1}, \dots, X_{p-2}-X_{p-1})-X_{p-1}\sigma_{k-1}(X_1-X_{p-1}, X_2-X_{p-1}, \dots, X_{p-2}-X_{p-1})=\\
&(-1)^{2k}X_{p-1}^k+(-1)^{2k-2+1}X_{p-1}^k=0.
\end{align*}
By Section \ref{powersum}, this completes the proof. It also shows for $m=p-1$ that the projectivity $\Lambda$ of $PG(m-2,\mathbb{K})$ defined by $\lambda$ preserves $V$. We point out that $\Lambda$ is an involutory homology. From the definition of $\lambda$, every point on the hyperplane $H_{p-1}$  of equation $X_{p-1}=0$, as well as, the point $C=(1:1:\cdots:2)$ are fixed by $\Lambda$. Therefore, $\Lambda$ is a homology with center $C$ and axis $H_{p-1}$. Furthermore,
$\Lambda^2$ the identity and hence $\Lambda$ is an involution.
\end{proof}

In the subsequent sections we further investigate the two extremal cases in positive characteristic, namely $m=5$ and $m=p-1$. Accordingly, from now on, $\mathbb{K}$ stands for the algebraic closure of the field $\mathbb{F}_p$ where $p\geq 7$ is a prime. Then, $V$ is defined over $\mathbb{F}_p$ and viewed as a curve defined over $\mathbb{K}$. The group $G$ is also defined over $\mathbb{F}_p$, and it preserves the set $\cX(\mathbb{F}_{p^i})$ of points of $V$ defined over $\mathbb{F}_{p^i}$ for every $i \ge 1$.

\section{The case of positive characteristic: $m=p-1$}
Throughout this section we always assume $m=p-1$.
\subsection{The action of $L\aut(V)$ on $V$}
First, we point out a result on the number of solutions of System (\ref{sy}).
\begin{theorem}
\label{redeith} System (\ref{sy}) modulo $p$ has as many as $(p-1)!$ solutions.
\end{theorem}
\begin{proof}
One can count the solutions of (\ref{sy}) modulo $p$ up to a non-zero constant factor by computing the number of points of $V$ over $\mathbb{F}_p$.   Since any primitive $(p-1)$-th roots of unity in $\mathbb{K}$ is in $\mathbb{F}_p$, Lemma \ref{lem10octAchar0} yields $|V(\mathbb{F}_{p})|=(p-2)!$, and the claim follows.
\end{proof}

\begin{proposition}\label{3orbiteI}
The short orbits of $L\aut(V)$ on $V$ are exactly $\Omega_\omega$ and $\Omega_\varepsilon\cup \Omega_\theta$, and they have lengths $(m-1)!$ and $m(m-2)!+m!/2$, respectively.
\end{proposition}
\begin{proof} Theorem \ref{redeith} shows that $\Omega_\omega$ consists of all $\mathbb{F}_p$-rational points of $V$. From the proof of Theorem \ref{th26062024A}, $L\aut(V)$ is defined over $\mathbb{F}_p$. Therefore a short orbit of $L\aut(V)$ is $\Omega_\omega$. Since $p$ is prime to $|\Omega_\omega|$, this yields that the order of the stabilizer of a point in $\Omega_\omega$ is divisible by $p$, i.e. $\Omega_\omega$ is a non-tame orbit.  Actually, $\Omega_\omega$ is the unique non-tame orbit of $L\aut(V)$ on $V$. In fact, since $p^2$ does not divide $L\aut(V)$, any two $p$-subgroup of $L\aut(V)$ are conjugate in $L\aut(V)$, and hence the fixed points of the $p$-subgroups of $L\aut(V)$ belong to the same orbit of $L\aut(V)$. From \cite[Theorem 11.56]{HKT}, $L\aut(V)$ has either one or two tame short orbits on $V$.

The former case is investigated first. Let $\Omega$ be the unique tame-orbit of $L\aut(V)$ on $V$. Choose a point $Q$ from $\Omega_\varepsilon$. Then $Q$ is fixed by a cyclic group of odd order $p-2$. From the proof of Theorem \ref{th26062024A}, $Q$ is also fixed by the involutory homology $\Lambda$. Therefore, the stabilizer of $Q$ in $L\aut(V)$ is a cyclic group of order $2(p-2)\mu$. Since the quotient curve $V/L\aut(V)$ is rational by Lemma \ref{lem19oct}, the Hurwitz genus formula applied to $L\aut(V)$ gives
$$2\gg(V)-2=(p-2)!(\kappa(p-1)-1)-\frac{p!}{2(p-2)\mu}$$
where $\kappa$ counts the ramification groups of $L\aut(V)$ at $P$. This together with Lemma \ref{lem11octE} yield
\begin{equation}
\label{eq29062024}
p^2-5p+4+\frac{p(p-1)}{\mu}=2\kappa(p-2)(p-1).
\end{equation}
 In particular, Equation \eqref{eq29062024} yields $\mu|(p-1)$ by $p\nmid \kappa$. Taking (\ref{eq29062024}) mod $p$, $4\kappa \equiv 4 \pmod{p}$ is obtained whence $\kappa\equiv 1 \pmod{p}$ follows. For $\kappa\ge p+1$, (\ref{eq29062024}) is impossible. Thus $\kappa=1$, and hence $\mu=1$, that is, $|L\aut(V)_Q|=2(p-2)$. Therefore, the unique short tame-orbit $\Omega$ of $L\aut(V)$ has size $(p-1)(p-3)!+(p-1)!/2$. Since $\Omega$ contains both $\Omega_\varepsilon$ and $\Omega_\theta$, this yields $\Omega=\Omega_\varepsilon \cup \Omega_\theta$.

 It remains to rule out the possibility of two short tame-orbits. Assume on the contrary that $\Omega_1$ and $\Omega_2$ are two distinct short non-tame orbits of $L\aut(V)$. We may suppose $\Omega_\varepsilon\subseteq \Omega_1$. From the Hurwitz genus formula applied to $L\aut(V)$,
$$2\gg(V)-2=(p-2)!(\kappa(p-1)-1)-\frac{p!}{2(p-2)\mu}+|\Omega_2|(|L\aut(V)|-1).$$
Since $|L\aut(V)|-1\ge \ha |L\aut(V)|$, and $(p-2)!((p-1)-1)>\ha p!/(p-2)$, this yields $4\gg(V)-4>|L\aut(V)|$. But this is impossible, as $4\gg(V)-4=2((p-2)(p-3)-4)(p-3)!<p!=|L\aut(V)|$.
\end{proof}
\begin{proposition}
\label{prop30062024} $L\aut(V)$ acts on $\Omega_\omega$ as a primitive permutation group.
\end{proposition}
\begin{proof} Let $P$ be a point in $\Omega_\omega$. Then the stabilizer $U$ of $P$ in $L\aut(V)$ contains a cyclic group $H$ of order $p-1$. From the proof of Proposition \ref{3orbiteI}, $U$ also contains a Sylow $p$-subgroup $S_p$ of order $p$. Since $|\Omega_\omega|=(p-3)!$ and $|L\aut(V)|=p!$, it follows that $|U|=p(p-1)$. Therefore $U=S_p\rtimes H$. Now, look at the natural action of $\rm{Sym}_p$ on a set $\Theta$ of size $p$. Since $|U|=p(p-1)$, $U$ acts as a $2$-transitive permutation group.  Take a subgroup $W$ of $S_p$ that contains $U$ properly. From the classical Burnside theorem \cite[ Chapter X, Section 151-154]{burns}, $W$ is an almost simple group, i.e. there exists a simple group $T$ such that $T\le W\le \aut(T)$ up to an isomorphism. Thus Guralnick's classification, see Result \ref{gu}, applies to $T$. Case (ii) is ruled out  by Result \ref{shar}. Cases (iii) are also impossible as neither $PGL(2,11)$ nor $\aut(M_{11})=M_{11}$ have a subgroup of order $11\cdot10=110$ whereas $\aut(M_{23})=M_{23}$ has no subgroup of order $23\cdot 22=506$. Therefore, Case (i) occurs and hence $W= \rm{Alt}_p$, or $W=\rm{Sym}_p$. But $\rm{Alt}_p$ contains no cyclic group of order $p-1$ whence $W=\rm{Sym}_p$ follows. 
\end{proof}

\subsection{Further results on $V(\mathbb{F}_{p^i})$.}

From the proof of Theorem \ref{redeith}, $V(\mathbb{F}_p)$ is the $G$-orbit $\mathcal{O}_\omega$ defined in Lemma \ref{lem10octAchar0}. Furthermore, Lemmas \ref{lem7oct} and \ref{lem8octU} have the following corollary.
\begin{lemma}
\label{lem170621} The $G$-orbit $\mathcal{O}_\varepsilon$ is contained in $V(\mathbb{F}_{p^i})$ but not in $V(\mathbb{F}_{p^j})$ for $j<i$, where $\mathbb{F}_{p^i}$ is the smallest subfield of $\mathbb{K}$ containing a $(p-2)$-th primitive root $\varepsilon$ of unity.
\end{lemma}
We are in a position to prove the following theorem.
\begin{theorem}\label{nopuntiFp2}
The curve $V$ has no proper $\mathbb{F}_{p^2}$-rational point, that is $V(\mathbb{F}_{p^2})=V(\mathbb{F}_{p})$.
\end{theorem}
\begin{proof}
By way of a contradiction, $|V(\mathbb{F}_{p^2})|>| V(\mathbb{F}_{p})|$ is assumed. Then, since $G$ takes $\mathbb{F}_{p^2}$-rational points to $\mathbb{F}_{p^2}$-rational points, there exists a
$G$-orbit $\Omega$ entirely contained in $V(\mathbb{F}_{p^2})\setminus V(\mathbb{F}_{p})$.

Assume first that $\Omega$ is a long orbit, that is $|\Omega|=(p-1)!$, and let $H$ be the stabilizer of $X_{p-1}, X_{p-2}, X_{p-3}$ in $G$. Then $H$ partitions $\Omega$ into $(p-1)(p-2)(p-3)$ long $H$-orbits. Since each $H$-orbit corresponds to a point of $\mathcal{X}=V/H$, it follows that
\begin{equation*}
|\mathcal{X}(\mathbb{F}_{p^2})\setminus \mathcal{X}(\mathbb{F}_{p})|\geq (p-1)(p-2)(p-3).
\end{equation*}
On the other hand, the St\"ohr-Voloch bound \cite{sv}, (see also \cite[Theorem 8.41]{HKT}) applied to $\cX$  gives
\begin{equation*}
2|\mathcal{X}(\mathbb{F}_{p^2})| \leq (2\mathfrak{g}(\cX)-2)+(p^2+2)(p-3)=(p-6)(p-3)+(p^2+2)(p-3)=(p-3)(p^2+p-4).
\end{equation*}
Therefore, $2(p-1)(p-2)(p-3)\le (p-3)(p^2+p-4)$. But then $p<7$, a contradiction.

Assume now that $\Omega$ is a short orbit. From Lemma \ref{3orbite}, the only possibility is that $\Omega=\Omega_\theta$ for some transposition $\theta\in G$. Since each $H$-orbit corresponds to a point of $\mathcal{X}$, Lemma \ref{lambda12} implies
\begin{equation*}
|\mathcal{X}(\mathbb{F}_{p^2})\setminus \mathcal{X}(\mathbb{F}_{p})|\geq \lambda_1+\lambda_2=(p-3)(p^2-6p+11).
\end{equation*}
This time the St\"ohr-Voloch bound gives
\begin{equation*}
2\cdot (p-3)(p^2-6p+11)\leq (p-3)(p^2+p-4),
\end{equation*}
which is only possible for $p=7$. However, a Magma  aided computation rules out this possibility.
\end{proof}
From Lemmas \ref{lem10octAchar0} and \ref{lem7oct}, $\Omega_\omega=V(\mathbb{F}_p)$ and $\Omega_\varepsilon \subset V(\mathbb{F}_{p^j})$ respectively, where $\mathbb{F}_{p^j}$ is the smallest subfield of $\mathbb{K}$ containing a primitive ($p-2$)-th root of unity. We prove an analog claim for $\Omega_\theta$. As pointed out in the proof of Lemma \ref{lem11octC},  $\Omega_\theta$ has a point $P$ whose last two coordinates are equal $1$.
\begin{lemma}
\label{lem22jun21} Let $P=(\xi_1,\xi_2,\ldots,\xi_{p-3},1,1)$ be a point of $\Omega_\theta$. Then $\xi_{p-3}\in \mathbb{F}_{p^{j}}$ for some $1<j\le p-3$. Furthermore, if $\xi_{p-3}\not \in  \mathbb{F}_{p^{j}}$ with $j<p-3$ then, up to a permutation of the indices $\{1,2,\ldots, p-4\}$,
$$x_j=x_{p-3}^{p^j}\quad j=1,2,\ldots p-4.$$
\end{lemma}
 \begin{proof} Since both $V$ and $\Omega_\theta$ are defined over $\mathbb{F}_p$, Lemma \ref{3orbite} shows that the Frobenius map $\Phi$  takes $P$ to the point $P^{(p)}\in \Omega_\theta$ where $P^{(p)}=(\xi_1^p,\xi_2^p,\ldots,\xi_{p-3}^p,1,1)$. Also, $\Phi^i$ takes $P$ to the point
 $$P^{(p^i)}=(\xi_1^{p^i},\xi_2^{p^i},\ldots,\xi_{p-3}^{p^i},1,1)$$
 of $\Omega_\theta$. To prove the first claim, assume on the contrary $\xi_{p-3}\not \in  \mathbb{F}_{p^{j}}$ for $j\le p-3$.  Then $\xi_{p-3},
 \xi_{p-3}^p,\ldots, \xi_{p-3}^{p^{p-3}}$ are pairwise distinct. On the other hand, from Lemma \ref{lem10giugno}, $\xi_{p-3}^{p^i}\in \{\xi_1,\xi_2,\ldots,\xi_{p-3}\}$ for any $i\ge 0$; a contradiction. To prove the second claim, we may assume that $\xi_{p-3}\in \mathbb{F}_{p^{p-3}}$. Then the previous argument shows that  $\{\xi_{p-3}, \xi_{p-3}^p,\ldots, \xi_{p-3}^{p^{p-3}}\}=\{\xi_1,\xi_2,\ldots,\xi_{p-3}\}$ whence the claim follows.
 \end{proof}
 Theorem \ref{the24jun} together with Lemma \ref{lem22jun21} have the following corollary.
 \begin{proposition}
 \label{lem24jun21} Let $\mathbb{F}_{p^j}$ be the subfield of $\mathbb{K}$ which is the splitting field of the polynomial $f(X)=X^{p-3}+2X^{p-4}+3X^{p-5}+\ldots+(p-3)X+p-2$. Then $\Omega_\theta \subset V(\mathbb{F}_{p^j})$ but $\Omega_\theta \nsubseteq  V(\mathbb{F}_{p^i})$ for $i<j$.
  \end{proposition}
The proof of the following theorem relies on Proposition \ref{lem24jun21}.
\begin{theorem}
\label{the181021}
Let $d$ be the smallest positive  integer such that $(p-2)\mid (p^d-1)$. Then $\Omega_\theta, \Omega_\epsilon \subset V(\mathbb{F}_{p^d})$ but $\Omega_\theta, \Omega_\epsilon \nsubseteq  V(\mathbb{F}_{p^i})$ for $i<d$.
\end{theorem}
\begin{proof}
By definition, $\mathbb{F}_{p^d}$ is the smallest extension of $\mathbb{F}_p$ containing a primitive $(p-2)$-th root of unity. So, the claim holds for $\Omega_\epsilon$. To complete the proof, consider the polynomial $$g(X)=\frac{X^{p-2}-1}{X-1}=X^{p-3}+X^{p-4}+\ldots+X+1,$$
whose splitting field is $\mathbb{F}_{p^d}$. We prove that $g(1-X)=f(X)$. Indeed
$$
g(1-X)=\sum_{i=0}^{p-3} \sum_{j=0}^i \binom{i}{j}(-X)^j,
$$
and for $i\in \{0,\ldots,p-3\}$ the coefficient of $X^i$  in $g(1-X)$ is $$(-1)^i\sum_{k=i}^{p-3}\binom{k}{i}=(-1)^i\sum_{k=0}^{p-3-i}\binom{i+k}{i}=(-1)^i\binom{p-2}{p-3-i}=(-1)^i\binom{p-2}{i+1}.$$ Thus,
$$
g(1-X)=\sum_{i=0}^{p-3}\binom{p-2}{i+1}(-1)^i X^i=\sum_{i=0}^{p-3}(p-2-i)X^i=f(X).
$$
Therefore the splitting field of $f(X)$ coincides with the splitting field of $g(X)$ and Proposition \ref{lem24jun21} yields the claim.
\end{proof}

\subsection{Non-classicality and Frobenius non-classicality of $V$}
\label{nc}
For $m=p-1$, Proposition \ref{prop10.03} has the following corollary.
\begin{proposition}
\label{prop10.03A} If $\mathbb{K}$ has characteristic $p$ then $V$ is contained in the Hermitian variety $\cH_{p-3}$  which is the intersection of the hyperplane $\Pi$ of equation $X_1 + X_2 + \ldots + X_{p-1}=0$ with the Hermitian variety $\cH_{p-2}$ of equation $X_1^{p+1} + \ldots+X_{p-1}^{p+1}=0.$
\end{proposition}
\begin{lemma}
\label{lem17ag21} For a point $P\in V$, let $\Pi_P$ be the tangent hyperplane to $\cH_{p-3}$ at $P$. Then $I(P,V\cap \Pi_P)\ge p$.
\end{lemma}
\begin{proof} Clearly, $\Pi_P$ is the intersection of the hyperplane $\Pi$ with the tangent hyperplane $\alpha_P$ to the Hermitian variety $\cH_{p-2}$ at $P$. Since $V$ is contained in $\Pi$, $I(P,V\cap \Pi_P)=I(P,V\cap \alpha_P)$ holds. Hence, it is enough to show that $I(P,V\cap \alpha_P)$ is at least $p$.

Let $P=(\xi_1:\cdots:\xi_{p-2}:\xi_{p-1}).$
Up to a change of coordinates, $\xi_{p-1}=1$ may be assumed. In the affine space ${\rm{AG}}(p-2,\mathbb{K})$ with infinite hyperplane $X_{p-1}=0$, $\cH_{p-2}$ has equation $X_1^{p+1}+X_2^{p+1}+\ldots + X_{p-2}^{p+1}+1=0$.  For every $i=1,\ldots,p-2,$ let    $x_i(t)=\xi_i+\rho_i(t)$ with $x_i(t),\rho_i(t)\in \mathbb{K}[[t]]$ and ${\rm{ord}}(\rho_i(t))\ge 1$  be a primitive representation of the unique branch of $V$ centered at $P$.
By Proposition \ref{prop10.03A}, $V$ is contained in $\cH$. Therefore,
$x_1(t)^{p+1}  +  x_2(t)^{p+1}  +  \ldots+x_{p-2}(t)^{p+1}+1$
vanishes in $\mathbb{K}[[t]].$ From this
$$\xi_1^p(\xi_1+\rho_1(t))+\xi_2^p(\xi_2+\rho_2(t))+\ldots +\xi_{p-2}^p(\xi_{p-2}+\rho_{p-2}(t))+1=t^pv(t),\,\,v(t)\in \mathbb{K}[[t]],{\rm{ord}}(v(t))\ge 0.$$
Since $\Pi_P$ has equation $\xi_1^{p}X_1+\xi_2^{p}X_2+\ldots +\xi_{p-2}^{p}X_{p-2}+1=0$ in ${\rm{AG}}(p-2,\mathbb{K})$, the claim follows.
\end{proof}
Since the dimension of ${\rm{PG}}(p-2,\mathbb{K})$ is smaller than $p$, Lemma \ref{lem17ag21} has the following corollary.
\begin{theorem}
\label{the17ag21} $V$ is a non-classical curve.
\end{theorem}
\begin{theorem}
\label{the191021} $V$ is a Frobenius non-classical curve.
\end{theorem}
\begin{proof} Assume on the contrary that $V$ is Frobenius classical. Then $0,1,\ldots,p-4$ are orders at a generically chosen point of $V$. Lemma \ref{lem17ag21} yields that the last order, $\varepsilon_{p-3}$, is equal to $p$. Therefore, Lemma \ref{lem17ag21} also yields that the osculating tangent hyperplane to $V$ at $P$ coincides with the tangent hyperplane $\Pi$ to Hermitian variety $\cH_p$. Since $\Pi$ passes through the Frobenius image of $P$, it follows that
$V$ is Frobenius non-classical, a contradiction.
\end{proof}

\subsection{Some results on the orders of $V$ at $\mathbb{F}_p$-rational points}
\label{secc1}
Since $G$ is transitive on $V(\mathbb{F}_p)$, the orders of $V$ are the same at every point $P\in V(\mathbb{F}_p)$. From Lemma \ref{lem1210}, such a point is $P=(1:\eta^{p-2}:\eta^{p-3}:\dots:\eta)$ for a primitive element $\eta$ of $\mathbb{F}_p$. From Lemma \ref{lem10octAchar0}, the stabilizer of $P$ in $G$ is a cyclic group of order $p-1$ generated by the projectivity $\sigma$ associated with the matrix
\begin{equation}\label{circulant}
M_\sigma=
\begin{pmatrix}
0 & 1 & 0 & &\cdots & 0\\
0 & 0 & 1 & &\cdots & 0\\
 &  &  & \ddots &  & \\
\vdots& & &\ddots & &\\
0 & 0 & 0 &\cdots & 0 & 1\\
1 & 0 & 0 &\cdots & 0 & 0\\
\end{pmatrix}.
\end{equation}
The eigenvalues of $M_\sigma$ are $\lambda_i=\eta^i$ for $i=0,\dots,p-2$. Moreover, then eigenvectors of $M_\sigma$ are ${\bf{w}}_i=(1,\eta^i,\eta^{2i},\dots,\eta^{(p-2)i})$. Thus the point of ${\rm{PG}}(p-2,\mathbb{F}_p)$  represented by ${\bf{w}}_i$  is in $V$ if and only if ${\rm{g.c.d.}}(i,p-1)=1$.  Therefore, $\sigma$ has as many as $\varphi(p-1)$ fixed points on $V$.

Let $\eta_i=\eta^{i}$ for $i=0,\ldots,p-2$, and consider the change of the projective frame from $(X_1:X_2:\ldots:X_{p-1})$ to $(Y_1:Y_2:\ldots:Y_{p-1})$ defined by
\begin{equation}\label{varchange}
\begin{cases}
X_1=Y_1+Y_2+\cdots+Y_{p-1};\\
X_2=Y_1+\eta_1Y_2+\eta_1^2Y_3\cdots+\eta_{1}^{p-2}Y_{p-1};\\
X_3=Y_1+\eta_2Y_2+\cdots+\eta_{2}^{p-2}Y_{p-1};\\
\vdots\\
X_{p-1}=Y_1+\eta_{p-2}Y_2+\cdots+\eta_{p-2}^{p-2}Y_{p-1}.
\end{cases}
\end{equation}
With this transformation, $P$ is taken to the fundamental point  $O=(0:0:\cdots:0 :1)$.
Let $R$ be the matrix whose rows are the vectors ${\bf{w}}_i$, $i=1,\dots,p-1$. Then $R^{-1}M_\sigma R$ is the diagonal matrix
\begin{equation*}
D=\begin{pmatrix}
\eta &0&&\cdots&0\\
0&\eta^2&0&\cdots&0\\
\vdots& &&\ddots &\\
0&\cdots&& 0&\eta^{2u}
\end{pmatrix},
\end{equation*}
where $2u=p-1$, and hence it is the matrix associated to $\sigma$ in the projective frame $(Y_1:Y_2:\ldots:Y_{p-1})$. Replacing $X_i$ by (\ref{varchange}) in \eqref{sy}, we obtain the equations of $V$ in the projective frame $(Y_1:Y_2:\ldots:Y_{p-1})$.
Now, we explicitly write down the equations defining $V$ in the projective frame $(Y_1:Y_2:\ldots:Y_{p-1})$. From the first equation in \eqref{sy}, we obtain $Y_1=0$. In fact, $\sum_{i=1}^{p-1}X_i=(p-1)Y_1+Y_2\sum_{j=0}^{p-2}\eta_1^j+\cdots+Y_{p-1}\sum_{j=0}^{p-2}\eta_{p-2}^j,$ and for $i=1\dots,p-2$ we have $\sum_{j=0}^{p-2}\eta_{i}^j=(\eta_i^{p-1}-1)/(\eta_i-1)=0$.
Therefore, $V$ is a projective variety of ${\rm{PG}}(p-3,\mathbb{F}_p)$ with projective frame $(Y_2:Y_3:\cdots :Y_{p-1})$.

Since $O$ is off the hyperplane of homogenous equation $Y_{p-1}=0$, a branch representation of the unique branch $\gamma$ of $V$ centered at $O$ has as components $y_2=y_2(t),\dots,y_{p-2}=y_{p-2}(t),y_{p-1}=1$. Here $y_i(t)\in \mathbb{F}_p[[t]]$. Furthermore, we can assume $ord(y_{p-2}(t))=1$, since $O$ is a simple point of $V$ and the tangent hyperplane to $V$ at $O$ does not contain the line of homogeneous equation $Y_2=0,\ldots, Y_{p-3}=0$.
Therefore,
\begin{equation}\label{ramo}
\begin{cases}
&y_2(t)= \alpha_{2,1}t+\alpha_{2,2}t^2+\dots\\
&y_3(t)= \alpha_{3,1}t+\alpha_{3,2}t^2+\dots\\
&\vdots \\
&y_{p-3}(t)=\alpha_{p-3,1}t+\alpha_{p-3,2}t^2+\dots\\
&y_{p-2}(t)=t\\
&y_{p-1}(t)=1.
\end{cases}
\end{equation}
Since the projectivity $\sigma$ preserves $V$ and fixes $O$, it also preserves $\gamma$. Moreover, $\sigma$ is associated to the above diagonal matrix $D$. This yields that an equivalent branch representation is given by
\begin{equation}\label{ramo2}
\begin{cases}
&\bar{y}_1(t)=0\\
&\bar{y}_2(t)=\eta^2( \alpha_{2,1}t+\alpha_{2,2}t^2+\dots)\\
&\bar{y}_3(t)=\eta^3( \alpha_{3,1}t+\alpha_{3,2}t^2+\dots)\\
&\vdots \\
&\bar{y}_{p-3}(t)=\eta^{p-3}(\alpha_{p-3,1}t+\alpha_{p-3,2}t^2+\dots)\\
&\bar{y}_{p-2}(t)=\eta^{p-2}t=\eta^{-1}t\\
&\bar{y}_{p-1}(t)=1.
\end{cases}
\end{equation}
Replacing $t$ by $\eta^{-1}t$, the equations of \eqref{ramo2} become
\begin{equation}\label{ramo3}
\begin{cases}
&\bar{y}_1(t)=0\\
&\bar{y}_2(t)=\eta^2( \eta\alpha_{2,1}t+\eta^2\alpha_{2,2}t^2+\dots)\\
&\bar{y}_3(t)=\eta^3( \eta\alpha_{3,1}t+\eta^2\alpha_{3,2}t^2+\dots)\\
&\vdots \\
&\bar{y}_{p-3}(t)=\eta^{p-3}(\eta\alpha_{p-3,1}t+\eta^2\alpha_{p-3,2}t^2+\dots)\\
&\bar{y}_{p-2}(t)=t.\\
\end{cases}
\end{equation}
Therefore \eqref{ramo} and \eqref{ramo3} are the same branch representation of $\gamma$, whence
\begin{equation*}
\begin{cases}
&\eta^3\alpha_{2,1}=\alpha_{2,1};\\
&\vdots\\
&\eta^{2+i}\alpha_{2,i}=\alpha_{2,i};\\
&\vdots\\
&\eta^{2u-1}\alpha_{2u-1,1}=\alpha_{2u-1,1};\\
&\vdots\\
&\eta^{2u-2+i}\alpha_{2u-2,i}=\alpha_{2u-2,i};\\
&\vdots\\
\end{cases}
\end{equation*}
From this $\alpha_{2,i}=0$ for $i< 2u-2$. More generally, $\alpha_{k,i}=0$ for $2\leq k\leq p-3$ and $i<p-1-k$.
Therefore, $ord(y_k(t))\geq p-1-k$ for $2\leq k\leq p-3$. Moreover, the only coefficients $\alpha_{k,i}\neq 0$ are among those verifying $i+k\equiv 0\pmod{p-1}$, where $2\leq k\leq p-3$ and $i\geq 1$.
Thus, the branch representation is
\begin{equation}\label{ramo3bis}
\begin{cases}
&y_2(t)= \alpha_{2,p-3}t^{p-3}+\alpha_{2,2p-4}t^{2p-4}+\dots+\alpha_{2,w(p-1)-2}t^{w(p-1)-2}+\dots\\
&y_3(t)= \alpha_{3,p-4}t^{p-4}+\alpha_{3,4u-3}t^{4u-3}+\dots+\alpha_{3,w(p-1)-3}t^{w(p-1)-3}+\dots\\
&\vdots \\
&y_{p-4}(t)=\alpha_{p-4,3}t^3+\alpha_{p-4,p+2}t^{p+2}+\dots+\alpha_{p-4,w(p-1)+3}t^{w(p-1)+3}+\dots\\
&y_{p-3}(t)=\alpha_{p-3,2}t^2+\alpha_{p-3,p+1}t^{p+1}+\dots+\alpha_{p-3,w(p-1)+2}t^{w(p-1)+2}+\dots\\
&y_{p-2}(t)=t\\
&y_{p-1}(t)=1.
\end{cases}
\end{equation}
We now show how to compute the remaining orders.\\
Recall that, in the new coordinates, the $k$-th equation in \eqref{sy} reads
\begin{equation}\label{nuoveEquazioni}
\sum_{j=1}^{p-1}\bigg(\sum_{i=2}^{p-2}\eta_{j-1}^{i-1}Y_i+\eta_{j-1}^{p-2}Y_{p-1}\bigg)^k,
\end{equation}
and that by the multinomial theorem
\begin{equation*}
\bigg(\sum_{i=2}^{p-2}s_i+v\bigg)^k=\sum_{k_2+\dots+k_{p-2}+l=k}\binom{k}{k_2,\dots,k_{p-2},l}\prod_{w=2}^{p-2}s_w^{k_w}\cdot v^l.
\end{equation*}

Replacing $Y_{i}$ by $y_i(t)=\sum_{s=1}^\infty\alpha_{i,(p-1)s-i}t^{(p-1)s-i}$ in the $k$-th equation, it is obtained
\begin{equation}\label{eqramo}
\begin{split}
&\sum_{j=1}^{p-1}\bigg(\sum_{i=2}^{p-2}\eta_{j-1}^{i-1}\big(\sum_{s=1}^\infty\alpha_{(p-1)s-i,i}t^{(p-1)s-i}\big)+\eta_{j-1}^{p-2}\bigg)^k=0.
\end{split}
\end{equation}

First, $ord(y_{p-3}(t))=2$, and, more precisely, $\alpha_{p-3,2}=2\neq 0$. This can be observed by looking at the quadratic term in $t$ in the $(p-3)$-th equation, after the substitution $Y_i=y_i(t)$.


By Proposition \ref{prop10.03}, $k\in\{1,2,\dots,p-3\}\cup\{p+1,p+2,\dots,2p-5\}$.\\
In order to compute the orders of $y_i(t)$, with $i<p-1-2$, let $k\in \{p+1,p+2,\dots,2p-5\}$, and write $k=2p-2-\tilde{k}$, $\tilde{k}\in\{3,\dots,p-3\}$.
\begin{rem}\label{obs1}
For any fixed $p$, the coefficients $\alpha_{p-1-\tilde{k},\tilde{k}}$, for $\tilde{k}\in\{3,\dots,p-3\}$, can be computed by taking into account the following constraints.
\end{rem}

The expansion of \eqref{eqramo} is of the form $$\sum_{j=1}^{p-1}\sum_{k_2+\cdots+k_{p-2}+l}\binom{k}{k_2,\dots,k_{p-2},l}\eta_{j-i}^{(2p-2)l}\prod_{i=2}^{p-2}\eta_{j-1}^{(i-1)k_i}t^{(p-1-i)k_i}.$$

Here, $k_i\leq \tilde{k}$ for terms of degree $\tilde{k}$ in $t$.
Moreover, since $k>p$  and
$$\binom{k}{k_1,k_2,\dots,k_{p-1-1},l}=\frac{(k)!}{k_1!k_2!\cdots k_{2u-1}!l!};$$ we see that $l\geq p$ is a necessary condition in order to have a non-zero multinomial coefficient. Further, $l+\sum_{i=1}^{p-2}k_i=\tilde{k}$, and each term of degree $\tilde{k}$ corresponds to a choice of the indices such that $\sum(p-1-i)k_i=\tilde{k}$.

Moreover, since for $u\neq 0
\pmod{p-1}$,
\begin{equation*}
\sum_{j=1}^{2u}\bigg(\eta^{j-1}_{u}\bigg)=
\sum_{j=0}^{2u-1}\eta^j_{u}=
\frac{\eta^{2u}_{u}-1}{\eta_{u}-1}=0,
\end{equation*}
every non-zero term must satisfy $\sum(i-1)k_i\equiv l\pmod{p-1}$.\\
To illustrate Remark \ref{obs1}, we show that $\alpha_{p-1-3,3}$
is non-zero, whereas $\alpha_{2,p-1-2}=0$.\\
\begin{itemize}
\item
Let $\tilde{k}=3$, $k=2p-5$. In the $k$-th equation, the terms of degree $\tilde{k}$ in $t$ are precisely the three terms corresponding to the choices $s=1$, $i=p-1-1$, $k_i=3$ and $l=2p-5-3$ (for $j\neq i$ it will be $k_j=0$), or $s=1$, $i=p-1-1$, $k_i=1$ and $l=2p-5-1$(for $j\neq i$, $k_j=0$), and $s=1$, $i_1=p-1-1$, $k_{i_1}=1$, $i_2=p-1-2$, $k_{i_2}=1$, $i_2=1$ and $l=2p-5-2$(for $j\neq i$, $k_j=0$).
$$\binom{2p-5}{k_2,\dots,k_{2u-1},l}=\frac{(2p-5)!}{k_2!\cdots k_{2u-1}!l!}$$
Therefore, the term of degree $3$ in the $(2p-5)$-th equation is
$$ \Bigg[ \binom{2p-5}{1,1,0,\dots,0,2p-5-2}\alpha_{p-1-2,2}+\binom{2p-5}{3,0,0\dots,0,2p-5-3}+\binom{2p-5}{0,0,1,0\dots,0,2p-5-1}\alpha_{p-1-3,3}\Bigg]t^3=$$
$$
=\Bigg[ \alpha_{p-1-2,2}\cdot (-5)\cdot (-6)+(-1)(-5)(-7)-5\alpha_{p-1-3,3}\Bigg]t^3
$$
Since $\alpha_{p-3,3}=2$, it follows $\alpha_{p-1-3,3}=5\neq 0$.

\item Let $\tilde{k}=p-3$, $k=(p+1)$.
Observe that, in this case, the factors in the multinomial theorem are of the form $\binom{p+1}{k_1,k_2,\dots,k_{2u-1},l}$ with $k_1+k_2+\dots+k_{2u-1}+l=p+1$. Also, the only non-zero coefficients are those with $k_i=p$ or $k_i=p+1$ for some $i$. Moreover, for $u\neq 0
\pmod{p-1}$,
\begin{equation*}
\sum_{j=1}^{2u}\bigg(\eta^{j-1}_{u}\bigg)=
\sum_{j=0}^{2u-1}\eta^j_{u}=
\frac{\eta^{2u}_{u}-1}{\eta_{u}-1}=0.
\end{equation*}
Therefore, for $k=p+1$, the possibly non-vanishing terms in equation \eqref{eqramo}, are those obtained for $i+\tilde{i}\equiv 2\pmod{2u}$, taking $\eta_{j-1}^{i-1}y_i(t)$ with multiplicity $p$ and $\eta_{j-1}^{\tilde{i}-1}y_{\tilde{i}}(t)$ with multiplicity $1$ or taking $\eta_{j-1}^{i-1}y_i(t)$ with multiplicity $1$ and $\eta_{j-1}^{\tilde{i}-1}y_{\tilde{i}}(t)$ with multiplicity $p$, and those
obtained for $i=m+1$ and $\eta_{j-1}^{i-1}y_2(t)$ with multiplicity $p+1$.\\
Therefore, the term of degree $2u-2$ in $t$, can only be obtained by taking $\eta_{j-1}^{2-1}y_2(t)$ with multiplicity $1$ and $\eta_{j-1}^{2u-1}$ with multiplicity $p$, namely it is $-\alpha_{2,2u-2}t^{2u-2}$, and therefore $\alpha_{2,2u-2}=0$. As a consequence, $ord(y_2(t))\geq 2p-4$.

\end{itemize}

\begin{proposition}
\label{pro211021} Each of the integers $1,2,3$ are intersection multiplicities $I(P,V\cap \pi)$ at any $\mathbb{F}_p$-rational point. Furthermore, the last order is at least $2p-4$.
\end{proposition}

\subsection{Case $p=7$}
From Section \ref{secc1}, $y_1(t)=0$, $y_5(t)=t$, and $ord(y_4(t))=2$, with $\alpha_{4,2}=2$, that is $y_4(t)=2t^2+\alpha_{4,4}t^4+\cdots$. The second equation reads
$$-y_3^3+y_2y_3y_4+4y_2^2y_5-y_5^3+y_4y_5+4y_3=0 ;$$
whence $\alpha_{3,3}=5$, that is $y_3(t)=5t^3+\alpha_{3,6}t^6+\cdots$. The eighth equation reads
$$y_2^7+y_2+y_3^7y_5+y_3y_5^7+y_4^8=0.$$
Therefore it must be $ord(y_2(t))=10$.\\
Thus, the order sequence of the curve $V$ at the origin is $(0,1,2,3,10)$.
\subsection{Case $p=11$} With the support of MAGMA the first terms of the branch expansions of the $y_i$ can be computed.
\begin{equation*}
\begin{split}
y_1(t)=0;\\
y_2(t)=7t^{18}+\cdots ;\\
y_3(t)=5t^7+6t^{17}+\cdots;\\
y_4(t)=3t^6+2t^{16}+\cdots;\\
y_5(t)=9t^5+t^{15}+\cdots;\\
y_6(t)=3t^4+0t^{14}+\cdots;\\
y_7(t)=5t^3+6t^{13}+\cdots;\\
y_8(t)=2t^2+4t^{12}+\cdots;\\
y_9(t)=t.
\end{split}
\end{equation*}
This shows that the order sequence of $V$ at the origin is $(0,1,2,3,4,5,6,7,18)$.
\section{Maximality of the $p$-characteristic analog of the Bring curve for $m=5$}
\section{The case of positive characteristic; $m=5$}
In this section $m=5$, that is, $V$ is the $p$-characteristic analog of the Bring curve. Therefore, $V$ has genus $4$ and is embedded in $\PG(4,\mathbb{K})$ as the complete intersection of an hyperplane, a quadratic and a cubic surface, both non-singular. The (homogeneous) function field $\mathbb{K}(V)$ is $F=\mathbb{K}(x_1,x_2,x_3,x_4,x_5)$ 
with
 \begin{equation}
\label{syAA}\left\{
\begin{array}{llll}
x_1 + x_2 +x_3+x_4+x_5=0;\\
x_1^2  +  x_2^2 + x_3^2+ x_4^2+ x_5^2=0;\\
x_1^3  +  x_2^3 + x_3^3+ x_4^3+ x_5^3=0.
\end{array}
\right.
\end{equation}
Our goal is to show that $V$ is an $\mathbb{F}_{p^2}$-maximal curve, that is the number of points of $V$ defined over $\mathbb{F}_{p^2}$ attains the Hasse-Weil upper bond $p^2+1+2\mathfrak{g}p=p^2+1+8p$ for infinitely many values of $p$. The essential tool for the proof is the Jacobian variety $J_V$ associated with $V$, with the following characterization of maximal curves due to Tate \cite[Theorem 2(d)]{tate} and explicitly pointed out by Lachaud \cite[Proposition 5]{la}: a curve defined over $\mathbb{F}_p$ of genus $\mathfrak{g}$ is an $\mathbb{F}_{p^2}$-maximal curve if and only if its Jacobian is $\mathbb{F}_{p^2}$-isogenous to the $\mathfrak{g}$-th power of an $\mathbb{F}_{p^2}$-maximal elliptic curve. To determine $J_V$
we use the following Kani-Rosen theorem \cite[Theorem B]{kr} about the decomposition of the Jacobian variety of an algebraic curve with respect to an automorphism group $G$ equipped by a partition, that is, the group $G$ has a family of subgroups $\{H_1, \dots, H_t\}$ such that $G=H_1\cup\cdots\cup H_t$ and $H_i\cap H_j=\{1\}$ for $1\le i<j\le t$.
\begin{theorem}[Kani-Rosen]\label{kanirosen}
Let $G$ be a finite automorphism group of an algebraic curve $\cX$. If $G$ is equipped by a partition, then the following isogeny relation holds:
\begin{equation}\label{kaniroseneq1}
J_{\cX}^{t-1}\times J_{\cX/G}^{|G|}\sim J_{\cX/{H_1}}^{h_1}\times\cdots\times J_{\cX/{H_t}}^{h_t},
\end{equation}
where $H_1, \dots, H_t$ are the components of the partition and $h_i=|H_i|$ for $i=1,\ldots t$.
\end{theorem}
In fact, the Kani-Rosen theorem applies to $G$ as $G\cong {\rm{Sym}}_5$ and ${\rm{Sym}}_5\cong \PGL(2,5)$ is equipped with a partition whose components form three conjugacy classes of lengths $15,6,10$, namely those consisting of all cyclic subgroups of order $4,5$ and $6$, respectively. Since $\aut(V)$ acts on the set of five coordinates $\{X_1,\ldots X_5\}$, representatives of the conjugacy classes are: $C_4=\langle (X_1,X_2,X_3,X_4)\rangle$, $C_5=\langle(X_1,X_2,X_3,X_4,X_5)\rangle$, and $C_6=\langle (X_1,X_2,X_3)(X_4,X_5)\rangle$, respectively. In our case, since $V/G$ is rational, Theorem \ref{kanirosen} reads
\begin{equation}\label{kaniroseneq}
J_{V}^{30}\sim J_{V/{C_4}}^{60}\times J_{V/{C_5}}^{30}\times J_{V/{C_6}}^{60}
\end{equation}
Therefore, $V$ is an $\mathbb{F}_{p^2}$-maximal curves if each $J_{V/C_i}$ with $i=4,5,6$ is either rational, or a $\mathbb{F}_{p^2}$-maximal elliptic curve.
We show first that several quotient curves of $V$ are rational.
\begin{proposition}\label{quozrazionale}
Each of the following quotient curves is a rational curve:
\begin{itemize}
\item $V/C_5$
\item $V/G_8$, where $G_8$ is a (dihedral) subgroup of $G$ of order $8$;
\item $V/G_{24}$, where $G_{24}$ is a subgroup (isomorphic to $\rm{Sym}_{4}$) of $G$ of order $24$;
\item $V/G_{12}$, where $G_{12}$ is a subgroup of $G$ of order $12$;
\item $V/G_{20}$, where $G_{20}$ is a subgroup of $G$ of order $20$.
\end{itemize}
\end{proposition}
\begin{proof}
We apply the Riemann-Hurwitz formula to each of the groups $C_5,G_8,G_{24},$ and $G_{12}$. For $i\in\{5,8,12,24\}$, let $\gg_i$ denote the genus of the corresponding quotient curve. From Theorem \ref{3orbite}, $G$ acts on $V$ with exactly $3$ short orbits, namely $\Omega_\omega$, $\Omega_\epsilon$ and $\Omega_\theta$, of length $24$, $30$ and $60$ respectively.
For $C_5$ the Riemann-Hurwitz formula reads
\begin{equation*}
6=10(\bar{\gg}_5-1)+\sum (5-|o_i|);
\end{equation*}
with $o_i$ running over the set of short orbits of $C_5$. Since $C_5$ has at least four fixed points on $\Omega_\omega$, it follows $\gg_5=0$. Also, since a group of order $20$ contains a subgroup of order $5$, this implies $\gg_{20}=0$.
For $G_8$ the Riemann-Hurwitz formula reads
\begin{equation*}
6=16(\bar{\gg}_8-1)+\sum (8-|o_i|);
\end{equation*}
with $o_i$ running over the set of short orbits of $G_8$. This implies $\bar{\gg}_8\le 1$. If equality holds then $G_8$ has a unique short orbit of length $2$. On the other hand, since neither $30$ nor $60$ is divisible by $8$,  $G_8$ has a short orbit in both $\Omega_\epsilon$ and $\Omega_\theta$, a contradiction which implies  $\bar{\gg}_8=0$. Finally, since $\rm{Sym}_4$ contains a subgroup of order $8$,  $\gg_{24}=0$ also holds.
For $G_{12}$, the Riemann-Hurwitz formula reads
\begin{equation*}
6=24(\bar{\gg}_{12}-1)+\sum (12-|o_i|);
\end{equation*}
with $o_i$ running over the set of short orbits of $G_{12}$. This implies $\bar{\gg}_{12}\le 1$. If equality holds then $G_{12}$ has a unique short orbit of length $12$, contained in $\Omega_\epsilon$. From the orbit-stabilizer theorem follows that $G_{12}$ contains an involution fixing a point in $\Omega_\epsilon$. This implies that the involution must be the product of two transposition, but it can be checked that the group $G_{12}$ does not contains any such element. It follows $\gg_{12}=0$.
\end{proof}
Next we show for every $p \ge 7$ that $J_{V/C_i}$ with $i=4,6$ is an elliptic curve, and that these two elliptic curves  are pairwise isogenous over $\mathbb{F}_{p^2}$.

The above defined $C_4$ can be viewed as an automorphism group of $F$ generated by the automorphism $(x_1,x_2,x_3,x_4,x_5)\mapsto (x_2,x_3,x_4,x_1,x_5)$. A Magma aided computation shows that each of following elements $b,c,d,\in F$ is left invariant by $C_4$:
$$
\begin{array}{lll}
b=-48 x_2 x_3^2 x_4^4 x_5^2 - 48 x_3^2 x_4^5 x_5^2 - 48 x_2 x_4^6 x_5^2 -48 x_4^7 x_5^2 -
24 x_2 x_3^2 x_4^3 x_5^3 - 24 x_2 x_3 x_4^4 x_5^3 -
    48 x_3^2 x_4^4 x_5^3 - \\
    48 x_2 x_4^5 x_5^3 - 24 x_3 x_4^5 x_5^3 -
    72 x_4^6 x_5^3 - 36 x_2 x_3^2 x_4^2 x_5^4 - 12 x_2 x_3 x_4^3 x_5^4 -
    48 x_3^2 x_4^3 x_5^4 - 84 x_2 x_4^4 x_5^4 - 24 x_3 x_4^4 x_5^4 -\\
    108 x_4^5 x_5^4 - 208/9 x_2 x_3^2 x_4 x_5^5 - 28/9 x_2 x_3 x_4^2 x_5^5 -
    370/9 x_3^2 x_4^2 x_5^5 - 668/9 x_2 x_4^3 x_5^5 - 82/9 x_3 x_4^3 x_5^5 -
    1046/9 x_4^4 x_5^5 - \\
    16/3 x_2 x_3^2 x_5^6 - 104/9 x_2 x_3 x_4 x_5^6 -
    152/9 x_3^2 x_4 x_5^6 - 44 x_2 x_4^2 x_5^6 - 118/9 x_3 x_4^2 x_5^6 -
    730/9 x_4^3 x_5^6 + 64/27 x_2 x_3 x_5^7 -\\ 8/3 x_3^2 x_5^7 -
    712/27 x_2 x_4 x_5^7 - 92/27 x_3 x_4 x_5^7 - 1306/27 x_4^2 x_5^7 -
    2128/243 x_2 x_5^8 + 32/27 x_3 x_5^8 - 5332/243 x_4 x_5^8 - \\1064/243 x_5^9,
\end{array}
$$
$$
\begin{array}{lll}
c=-67/3 x_2 x_3^2 x_4^2 x_5^4 + 1/3 x_2 x_3 x_4^3 x_5^4 + 68/3 x_3^2 x_4^3 x_5^4 -
    67/3 x_2 x_4^4 x_5^4 + 67/3 x_4^5 x_5^4 - 11 x_2 x_3 x_4^2 x_5^5 +
    1/6 x_3^2 x_4^2 x_5^5 - \\11 x_2 x_4^3 x_5^5 + 23/2 x_3 x_4^3 x_5^5 +
    67/6 x_4^4 x_5^5 - 68/9 x_2 x_3^2 x_5^6 + 2 x_2 x_3 x_4 x_5^6 +
    86/9 x_3^2 x_4 x_5^6 - 217/9 x_2 x_4^2 x_5^6 + \\
    5/18 x_3 x_4^2 x_5^6 +
    439/18 x_4^3 x_5^6 + 272/81 x_2 x_3 x_5^7 + 578/81 x_3^2 x_5^7 -
    790/81 x_2 x_4 x_5^7 - 97/81 x_3 x_4 x_5^7 + 107/6 x_4^2 x_5^7 -\\
    116/81 x_2 x_5^8 + 632/81 x_3 x_5^8 + 217/81 x_4 x_5^8 + 554/81 x_5^9,
\end{array}
$$
$$
\begin{array}{lll}
d=72 x_2 x_3^2 x_4^5 x_5 + 72 x_2 x_4^7 x_5 + 54 x_2 x_3^2 x_4^4 x_5^2 +
    36 x_2 x_3 x_4^5 x_5^2 + 18 x_3^2 x_4^5 x_5^2 + 90 x_2 x_4^6 x_5^2 +
    18 x_4^7 x_5^2 + 63 x_2 x_3^2 x_4^3 x_5^3 + \\
    27 x_2 x_3 x_4^4 x_5^3 +
    36 x_3^2 x_4^4 x_5^3 + 144 x_2 x_4^5 x_5^3 + 9 x_3 x_4^5 x_5^3 +
    45 x_4^6 x_5^3 + 178/3 x_2 x_3^2 x_4^2 x_5^4 + 9 x_2 x_3 x_4^3 x_5^4 +
    16 x_3^2 x_4^3 x_5^4 + \\ 154 x_2 x_4^4 x_5^4 + 55/3 x_3 x_4^4 x_5^4 +
    142/3 x_4^5 x_5^4 + 50/3 x_2 x_3^2 x_4 x_5^5 + 24 x_2 x_3 x_4^2 x_5^5 +
    29 x_3^2 x_4^2 x_5^5 + 298/3 x_2 x_4^3 x_5^5 +\\
    8/3 x_3 x_4^3 x_5^5 +
    209/3 x_4^4 x_5^5 + 52/9 x_2 x_3^2 x_5^6 - 2/9 x_2 x_3 x_4 x_5^6 +
    110/9 x_3^2 x_4 x_5^6 + 613/9 x_2 x_4^2 x_5^6 + 74/9 x_3 x_4^2 x_5^6 +\\
    139/3 x_4^3 x_5^6 - 208/81 x_2 x_3 x_5^7 + 532/81 x_3^2 x_5^7 +
    2260/81 x_2 x_4 x_5^7 + 226/27 x_3 x_4 x_5^7 + 2738/81 x_4^2 x_5^7 +
    4 x_2 x_5^8 + \\
    208/81 x_3 x_5^8 + 1460/81 x_4 x_5^8 + 676/81 x_5^9.
\end{array}
$$
Moreover, let $b_1=b/c,d_1=d/c$. From the above computation,
\begin{equation}
\label{quoc4}   135b_1^3-360b_1^2+240b_1+256+256d_1^2=0.
\end{equation}
Let $F_1$ be the subfield of $F$ generated by $b_1,d_1$. Then $F_1$ is elliptic and
the linear map $(b_1,d_1)\mapsto (x,y)$ with $x=-256/135b_1$, $y=-65536/18225d_1$, gives an  equation 
$$y^2 = x^3 + 2^{11}3^{-4}5^{-1} x^2 +2^{20}3^{-8}5^{-2}  x - 2^{32}3^{-12}5^{-4}$$
for $F_1$. Furthermore, $F_1$ is the fixed field $F^{C_4}$ of $C_4$. Indeed, since both $b_1$ and $d_1$ are fixed by $C_4$, $F_1 \subseteq F^{C_4}$ holds. On the other hand, if equality does not hold then Galois theory yields that $F_1$ is the fixed field of a subgroup $N$ of $G$ strictly containing $C_4$. Since $C_4$ is cyclic and $G\cong \rm{Sym}_5$, this yields that the order of $N$ is either $8$, $12$, $20$ or $24$. But then $V/N$ is rational by Proposition \ref{quozrazionale}. Therefore, $F_1=F^{C_4}$.

The quotient curve $V/C_6$ can be investigated in a similar way. The subfield $F^{C_6}$ of $F$ is generated by $A,B$ with
\begin{equation}
\begin{array}{lll}
5585034240000A^4 + 23225726880000A^3B + 27897294510000A^2B^2 +\\
    7952734845000AB^3 + 1056082140000B^4 +  13606338560000A^2B +\\
    28775567360000AB^2 + 6849136640000B^3 + 11767644160000B^2=0.
\end{array}
\end{equation}
The birational map
$$\begin{array}{lll}
(A,B)\mapsto(2^{19}3^{-6}5^{-2} AB  + 2^{20} 3^{-6}5^{-2} B^2,2^{32}37\cdot3^{-12}5^{-4}  A^2 + \\2^{30}313\cdot 3^{-12}5^{-4}  AB +
    2^{29}149\cdot   3^{-12}5^{-4} B^2 +\\ 2^{38}  3^{-12}5^{-4} B, -A^2 - 4AB - 4B^2),
\end{array}
$$
is a birational isomorphism over $\mathbb{F}_p$ to the elliptic curve $E_2$ with equation
\begin{equation}
y_2^2 = x_2^3 - 2^{20}3^{-8}5^{-2}71 x_2^2 -
2^{43}3^{-16}5^{-4}41 x_2 - 2^{64}23 \cdot
3^{-24}5^{-6}.
\end{equation}
Since, by Proposition \ref{quozrazionale}, the function field of the quotient curve of any subgroup of $G$ properly containing $C_6$ is rational, it follows that $\mathbb{K}( A,B)=\mathbb{K}(x_2,y_2)$ is the function field of $V/C_6$.

Furthermore, the elliptic curves $E_1$ and $E_2$ are isogenous over $\mathbb {F}_p$  via the isogeny
$$
\begin{array}{lll}
(x , y ) \mapsto ((2^{10}3^{-4}x^3 +   2^{20}\cdot 17\cdot3^{-8} 5^{-2} x^2 +
    2^{30}31\cdot3^{-12}\cdot 5^{-3}   x + 2^{40}11 \cdot 3^{-16}5^{-4})\\ / (x^2 - 2^{11}3^{-4}5^{-1} x +
    2^{20}3^{-8} 5^{-2}) , (2^{15}3^{-6} x^3y - 2^{25}3^{-9}5^{-1} x^2y -
    2^{35}13\cdot 3^{-14}5^{-2}  xy - \\ 2^{45}53\cdot 3^{-18}5^{-4}  y) / (x^3 -
     2^{10}3^{-3}5^{-1} x^2 + 2^{20}3^{-7}5^{-2} x - 2^{30}3^{-12}5^{-3}) ).
    \end{array}
$$
Observe that, since $p\geq 7$ the denominators do not vanish.

Summing up, the following theorem holds.
\begin{theorem}
\label{the27nov21}
For $p\geq 7$, the Jacobian variety $J_V$ of $V$ has a decomposition over $\mathbb{F}_p$ of the form
$J_V\sim E^4$
where $E=E_1$ is the elliptic curve of Weierstrass equation
$$y^2 = x^3 + 2^{11}3^{-4}5^{-1} x^2 +
 2^{20}3^{-8}5^{-2}  x - 2^{32}3^{-12}5^{-4}   .$$
\end{theorem}
Since $E$ is maximal over $\mathbb{F}_{p^2}$ for infinitely many primes \cite{elkies}, Theorem \ref{the27nov21} has the following corollary.
\begin{corollary}
$V$ is a $\mathbb{F}_{p^2}$-maximal curve for infinitely many primes $p$.
\end{corollary}
\begin{rem}\em{
The primes $p\leq 10000$ such that $V$ is $\mathbb{F}_{p^2}$-maximal are
$$p=29, 59, 149, 239, 269, 839, 1439, 1559, 2789, 2909, 4079, 4799, 5519, 6959, 8069, 8819, 9479, 9749.$$}
\end{rem}
\section{Connections with classical results}
\subsection{Connection with R\'edei's work on the Minkowski conjecture}
\label{secredei}
From \cite[Theorem 2.1]{szabo} and \cite[Lemma 6.1]{ln}, every solution $(\xi_1,\ldots,\xi_p)$ with $\xi_i\in \mathbb{F}_p$ also satisfies the diagonal equations $X_1^{k}  +  X_2^{k}  +  \ldots+X_{p}^{k}=0$ for $k=\ha(p+1),\ldots p-3$.

Now, let W be the algebraic variety of $\textrm{PG}(p-1,\mathbb{K})$ associated with the system
\begin{equation}
\label{syconp}\left\{
\begin{array}{llll}
X_1 + X_2 + \ldots + X_{p}=0;\\
X_1^2  +  X_2^2 + \ldots + X_{p}^2=0;\\
\cdots\cdots\\
\cdots\cdots\\
X_1^{p-3}  +  X_2^{p-3}  +  \ldots+X_{p}^{p-3}=0.
\end{array}
\right.
\end{equation}
of diagonal equations. Clearly $E=(1:1:\cdots: 1)$ is a point of $W$. Furthermore, any $\ell$ line through $E$ and another point $P\in W$ is entirely contained in $W$. In fact, if $P=(a_1:\cdots:a_p)$ then the points on $\ell$ distinct from $E$ are $Q=(\lambda+a_1:\cdots:\lambda+a_p)$ with $\lambda \in \mathbb{K}$, and a straightforward computation, $(\lambda+a_1,\ldots, \lambda+a_p)$ is also a solution of (\ref{syconp}). Since the hyperplane $\Pi$ of equation $X_p=0$ does not contain $E$, the algebraic variety cut out on $W$ by $\Pi$ is associated with (\ref{sy}) for $m=p-1$. Since this variety is $V$, Theorem \ref{redeith} and Lemma \ref{lem1210} show that the points of $W$ over $\mathbb{F}_p$ are $E$ together with the points $(\xi_1:\cdots: \xi_p)$ such that $[\xi_1,\ldots,\xi_p]$ is a permutation of the elements of $\mathbb{F}_p$. This gives a geometric interpretation for the result of R\'edei \cite{redei} and of Wang, Panico and Szab\'o \cite{redei} reported in the Introduction.

\subsection{Connection with classical geometry via Serre's work}
Let $\mathbb{K}(V)$ be the function field of $V$. Then $\mathbb{K}(V)=\mathbb{K}(x_1,\ldots,x_m)$ where $(x_1,\ldots,x_m)$ is a solution of (\ref{sy}), or equivalently of (\ref{sysA}). Let $s:=\sigma_{m-1}(x_1,\ldots,x_m)$ and $u:=\sigma_m(x_1,\ldots,x_m)$. Then $x_1,\ldots,x_m$ are the roots of the polynomial $F(Y)=Y^m+(-1)^{m-1}sY+(-1)^mu$. Replacing $Y$ in $F$ by $u/sZ$ gives $G(Z)=Z^m-Z^{m-1}-v$ with $v=u^{m-1}/s^m$. Therefore, $\mathbb{K}(V)$ is the Galois closure of the polynomial $Z^m-Z^{m-1}-v$. As Serre \cite[Chapter 4]{serA} pointed out in zero characteristic, $G(Z)$ has a cycle of order $m$ at $\infty$, another of order $m-1$ at $0$, and a transposition at $1-1/m$, and hence it has ramification type $(m,m-1,2)$. This holds true in our case, as $p>m$. Furthermore, using Serre's observation together with some classical results on Galois extensions one can show that $V$ has the properties shown in Theorem \ref{the2502} and in Lemmas \ref{lem11octE}; see \cite[Section 4.4]{serA}. Concerning the action of $\rm{Sym}_m$ on $V$, Lemma \ref{lem11octE} can also be proven by using Serre's observation together with \cite{vW}. However, the study of $\aut(V)$  appears to require much more knowledge about the geometry of $V$ and from group theory; see the proofs of Theorems \ref{the171021} and \ref{th26062024A}.

\section*{Acknowledgements}
This research was partially supported by the Italian National Group for
Algebraic and Geometric Structures and their Applications (GNSAGA - INdAM). The research of M. Timpanella was funded by the Irish Research Council, grant n. GOIPD/2021/93.


\begin{thebibliography}{99}
 \bibitem{AB} S.H. Alavi, T.C. Burness, Large subgroups of simple groups
\emph{J. Algebra} {\bf{421}} (2015), 187-233.
 \bibitem{AR} R. Auffarth, S. Rahausen, Galois subspaces for the rational normal curve,
\emph{Comm. Algebra} {\bf{49}} (2021),  3635-3644.
 \bibitem{bgkm} D. Bartoli, M. Giulietti, M.Q. Kawakita, M. Montanucci, New examples of maximal curves with low genus, \emph{Finite Fields Appl.} {\bf{68}} (2020), article 101744.
\bibitem{BN} H.W. Braden, T.P. Northover, Timothy P. Bring's curve: its period matrix and the vector of Riemann constants, \emph{SIGMA Symmetry Integrability Geom. Methods Appl.} {\bf{8}} (2012), Paper 065, 20 pp.
\bibitem{chen} Ri-Xiang Chen, On two classes of regular sequences,
\emph{J. Commut. Algebra} {\bf{8}} (2016),  29-42.
\bibitem{burns} W. Burnside, \emph{Theory of groups of finite order}, second edition, Cambridge University Press, 2013. 
\bibitem{CKW} A. Conca, C. Krattenthaler, J. Watanabe,
Regular sequences of symmetric polynomials,
\emph{Rend. Semin. Mat. Univ. Padova} {\bf{121}} (2009), 179-199.
\bibitem{DM}J.D. Dixon, B. Mortimer, \emph{Permutation Groups}, Springer, 1996.
\bibitem{elkies}N.D. Elkies, The existence of infinitely many supersingular primes for every elliptic curve over $\mathbb{Q}$, \emph{Invent. Math.}
89, 561–567 (1987).
\bibitem{DZ} R. Dvornicich, U. Zannier,
Newton functions generating symmetric fields and irreducibility of Schur polynomials,
\emph{Adv. Math.} {\bf{222}} (2009), 1982-2003.
\bibitem{FS} R. Fr\"oberg, B. Shapiro, On Vandermonde varieties,
\emph{Math. Scand.} {\bf{119}} (2016), 73-91.
\bibitem{Fu} S. Fukasawa, Galois lines for the Artin-Schreier-Mumford curve, \emph{Finite Fields Appl.} {\bf{75}} (2021), Paper No. 101894, 10 pg.
\bibitem{FH} S. Fukasawa, K. Higashine, Galois lines for the Giulietti-Korchm\'aros curve,
\emph{Finite Fields Appl.} {\bf{57}} (2019), 268-275.
\bibitem{GGW} F. Galetto, A.V. Geramita, D.L. Wehlau,
Degrees of regular sequences with a symmetric group action,
\emph{Canad. J. Math.} {\bf{71}} (2019), 557-578.
\bibitem{H} K. Higashine, A criterion for the existence of a plane model with two inner Galois points for algebraic curves, \emph{Hiroshima Math. J.} {\bf{51}} (2021), 163-176.
\bibitem{HKT} J.W.P.~Hirschfeld,
G.~Korchm\'aros and F.~Torres, \emph{Algebraic Curves Over a Finite Field}, Princeton Univ. Press, Princeton and Oxford, 2008.
\bibitem{gklm} M. Giulietti, M.Q. Kawakita, S. Lia, M. Montanucci, An $\mathbb{F}_{p^2}$-maximal Wiman’s sextic and its automorphisms, \textit{Adv. Geom.}  arXiv:1805.06317.
\bibitem{gkm} M. Giulietti,  G. Korchm\'aros, M. Montanucci, Maximal Curves over Finite Fields, Past, Present and Future, in Proc. of the international conference \emph{Curves over finite fields, Past, Present and Future} dedicated to J.P. Serre, pg.30, 2021.
\bibitem{gur}  R. M. Guralnick, Subgroups of prime power index in a simple group, \emph{J. Algebra} {\bf{81}} (1983), 304-311.
\bibitem{kr} E. Kani, M. Rosen, Idempotent relations and factors of Jacobians, \emph{ Math. Ann.} {\bf{284}} (1989) 307–327.
\bibitem{KLT} G. Korchm\'aros, S. Lia, M. Timpanella, Curves with more than one inner Galois point, \emph{J.  Algebra}, {\bf{566}} (2021), 374-404.
\bibitem{PV} G. Korchm\'aros, S. Lia, M. Timpanella, A generalization of Bring's curve in any characteristic, preliminary version,
\bibitem{la} G. Lachaud, Sommes d' Eisenstein et nombre de points de certaines courbes algébriques sur les corps finis, {\emph{C.R. Acad. Sci. Paris, Sér. I}} {\bf{305}} (1987) 729-732.
\bibitem{ln} R. Lidl -H. Niederreiter, \emph{Finite Fields}, Cambridge Univ. Press, 1983
\bibitem{KM} N. Kumar, I. Martino,
Regular sequences of power sums and complete symmetric polynomials,
\emph{Matematiche (Catania)} {\\bf{67}} (2012),  103-117.
\bibitem{MSW} H. Mel\'anov\'a, B. Sturmfels, R. Winter, Recovery from Power Sums, arXiv:2106.13981v1 [math.AG] 26 Jun 2021.
\bibitem{Sei} A. Seidenberg, \emph{Elements of the theory of algebraic curves}, Addison-Wesley Publishing Co., Reading, Mass.-London-Don Mills, Ont. 1968
\bibitem{sha} I. R. Shafaverich, \emph{Basic Algebraic Geometry 1: Varieties in Projective Space}, Springer, Heidelberg, third edition, 2013.
\bibitem{sti} H.~Stichtenoth,
\"Uber die Automorphismengruppe eines algebraischen
Funktionenk{\"o}rpers von Primzahl- charakteristik. I. Eine Absch{\"a}tzung
der Ordnung der Automorphismengruppe, \emph{Arch. Math.} {\bf 24} (1973),
527--544.
\bibitem{redei} L. R\'edei, \emph{Lacunary Polynomials over Finite Fields}, North-Holland Publishing Co., Amsterdam-London; American Elsevier Publishing Co., Inc., New York, 1973.
\bibitem{serA} J.P. Serre, \emph{Topics in Galois Theory} (Research Notes in Mathematics) 2nd Edition. Taylor and Francis, 2008.
\bibitem{ser} J.P. Serre, {\emph{Rational points on curves over finite fields}}, handwritten notes of Serre's Harward course 1985 by F. Gouvea, published in
{\em{Documents Math\'ematiques (Paris)}} {\bf{18}}. Soci\'et\'e Math\'ematique de France, Paris, 2020 with contributions of E. Howe, J. Oesterl\'e and C. Ritzenthaler.
\bibitem{sv} K.O. St\"ohr, J.F. Voloch, Weierstrass points and curves over finite fields, \emph{Proc. Lond. Math. Soc.} {\bf{52}} (1986), 1-19.
\bibitem{tate} J. Tate, Endomorphisms of Abelian varieties over finite fields, \emph{Invent. Math.} {\bf{2}} (1966), 134–144.
\bibitem{Vog} W. Vogel, \emph{Lectures on results on B\'ezout's Theorem.} Lecture Notes, Tata Institute of
Fundamental Research of Bombay. (Notes by D. P. Patil). Berlin-Heidelberg-New York, Springer. 1984 I, No. {\bf{74}}.
\bibitem{voloch} F. Rodríguez Villegas, F. Voloch, D. Zagier,  Constructions of plane curves with many points, \emph{Acta Arith.} {\bf{99}}, 85–96 (2001).
\bibitem{vW} B.L. van der Waerden, Die Zerlegungs- und Tr\"agheitsgruppe als Permutationsgruppen, \emph{ Math. Zeitschrift} {\bf{111}}, 731-733.
\bibitem{szabo} C. Ward, V.D. Panico, S. Szab\'o, Elementary proofs of two results of R\'edei, \emph{J. Algebra Comput. Appl.} {\bf{2}} (2012),  1-6.
\bibitem{AW} A. Wagner,  Groups generated by elations, \emph{Abhandlungen aus dem Mathematischen Seminar der Universit\"at Hamburg} {\bf{41}}, 1974, 190-205.
\bibitem{RW} R.A. Wilson, \emph{The Finite Simple Groups}, Springer Verlag 2009.
\bibitem{ZS} A.E. Zalesski$\breve{i}$, V.N. Sere$\hat{z}$kin, Finite linear groups generated by reflections, (Russian) {\emph{Izv. Akad. Nauk SSSR Ser. Mat.}} {\bf{44}} (1980), 1279-1307, translated in {\emph{Math. USSR. Izvestijia}} {\bf{17}} (1981),  477-503.

\end{thebibliography}
\end{document}